\documentclass{amsart}

\usepackage{graphicx}
\usepackage[centertags]{amsmath}
\usepackage{amsfonts}
\usepackage{amssymb}
\usepackage{amsthm}
\usepackage{newlfont}
\usepackage{hyperref}
\usepackage{epic}
\usepackage{eepic}
\usepackage{pstricks}
\usepackage{mathrsfs} 
\usepackage{yfonts} 

\newtheorem*{thm*}{Theorem}
\newtheorem{thm}{Theorem}[section]
\newtheorem{cor}[thm]{Corollary}
\newtheorem{lem}[thm]{Lemma}
\newtheorem{prop}[thm]{Proposition}

\theoremstyle{definition}
\newtheorem{defn}[thm]{Definition}

\theoremstyle{remark}

\numberwithin{equation}{section}

\DeclareSymbolFont{bbold}{U}{bbold}{m}{n}
\DeclareSymbolFontAlphabet{\mathbbold}{bbold}

\newcommand{\N}{\mathbb{N}}
\newcommand{\Z}{\mathbb{Z}}
\newcommand{\R}{\mathbb{R}}
\newcommand{\id}{\mathrm{id}}
\newcommand{\acts}{\curvearrowright}
\newcommand{\Stab}{\mathrm{Stab}}

\newcommand{\pmp}{p{$.$}m{$.$}p{$.$}}
\newcommand{\dbar}{d}
\newcommand{\dHam}{d_{\mathrm{Ham}}}

\newcommand{\pv}{\bar{p}}
\newcommand{\qv}{\bar{q}}

\newcommand{\cF}{\mathcal{F}}
\newcommand{\cJ}{\mathscr{I}}
\newcommand{\cP}{\mathcal{P}}
\newcommand{\cQ}{\mathcal{Q}}
\newcommand{\nm}{\mathcal{N}}

\newcommand{\dom}{\mathrm{dom}}
\newcommand{\rng}{\mathrm{rng}}

\newcommand{\salg}{\sigma \text{-}\mathrm{alg}}
\newcommand{\ralg}{\sigma \text{-}\mathrm{alg}^{\text{red}}}
\newcommand{\sH}{\mathrm{H}}
\newcommand{\rh}{h^{\mathrm{Rok}}}
\newcommand{\Int}[2]{\mathcal{I}_{#1}(#2)}
\newcommand{\Bnd}[2]{\partial_{#1}(#2)}
\newcommand{\Borel}{\mathcal{B}}
\newcommand{\ksh}{h^{\mathrm{KS}}}
\renewcommand{\:}{\,:\,}
\newcommand{\res}{\restriction}

\begin{document}

\title[Krieger's finite generator theorem for countable groups I]{Krieger's finite generator theorem for actions of countable groups I}
\author{Brandon Seward}
\address{Courant Institute of Mathematical Sciences, New York University, 251 Mercer Street, New York, NY 10003, U.S.A.}
\email{b.m.seward@gmail.com}

\begin{abstract}
For an ergodic {\pmp} action $G \acts (X, \mu)$ of a countable group $G$, we define the Rokhlin entropy $\rh_G(X, \mu)$ to be the infimum of the Shannon entropies of countable generating partitions. It is known that for free ergodic actions of amenable groups this notion coincides with classical Kolmogorov--Sinai entropy. It is thus natural to view Rokhlin entropy as a close analogue to classical entropy. Under this analogy we prove that Krieger's finite generator theorem holds for all countably infinite groups. Specifically, if $\rh_G(X, \mu) < \log(k)$ then there exists a generating partition consisting of $k$ sets.
\end{abstract}
\maketitle

\section{Introduction}

Let $(X, \mu)$ be a standard probability space, meaning $X$ is a standard Borel space with Borel $\sigma$-algebra $\Borel(X)$ and $\mu$ is a Borel probability measure. Let $G$ be a countable group, and let $G \acts (X, \mu)$ be a probability-measure-preserving ({\pmp}) action. For a collection $\xi \subseteq \Borel(X)$, we let $\salg_G(\xi)$ denote the smallest $G$-invariant $\sigma$-algebra containing $\xi$. A countable Borel partition $\alpha$ is \emph{generating} if $\salg_G(\alpha) = \Borel(X)$ (equality modulo $\mu$-null sets). The \emph{Shannon entropy} of a countable Borel partition $\alpha$ is
$$\sH(\alpha) = \sum_{A \in \alpha} - \mu(A) \cdot \log(\mu(A)).$$
A \emph{probability vector} is a finite or countable ordered tuple $\pv = (p_i)$ of positive real numbers which sum to $1$ (a more general definition will appear in Section \ref{SECT PRELIM}). We write $|\pv|$ for the length of $\pv$ and $\sH(\pv) = \sum - p_i \cdot \log(p_i)$ for the Shannon entropy of $\pv$.

Countable Borel partitions, and generating partitions in particular, have long played an important role in classical entropy theory. Generating partitions greatly simplify the definition and computation of Kolmogorov--Sinai entropy, and the proofs of two of the most well known results in entropy theory, Sinai's factor theorem and Ornstein's isomorphism theorem, relied upon deep, intricate constructions in which partitions played a starring role. Furthermore, generating partitions are more than merely a tool in entropy theory, but in fact are intimately connected with the notion of entropy itself. This fact is demonstrated by the following fundamental theorems of Rokhlin and Krieger.

\begin{thm*}[Rokhlin's generator theorem \cite{Roh67}, 1967]
If $\Z \acts (X, \mu)$ is a free ergodic {\pmp} action then its Kolmogorov--Sinai entropy $\ksh_\Z(X, \mu)$ satisfies
$$\ksh_\Z(X, \mu) = \inf \Big\{ \sH(\alpha) \: \alpha \text{ is a countable generating partition} \Big\}.$$
\end{thm*}

\begin{thm*}[Krieger's finite generator theorem \cite{Kr70}, 1970]
If $\Z \acts (X, \mu)$ is a free ergodic {\pmp} action and $\ksh_\Z(X, \mu) < \log(k)$ then there exists a generating partition $\alpha$ consisting of $k$ sets.
\end{thm*}

Both of the above theorems were later superseded by the following result of Denker.

\begin{thm*}[Denker \cite{De74}, 1974]
If $\Z \acts (X, \mu)$ is a free ergodic {\pmp} action and $\pv$ is a finite probability vector with $\ksh_\Z(X, \mu) < \sH(\pv)$, then for every $\epsilon > 0$ there is a generating partition $\alpha = \{A_0, \ldots, A_{|\pv|-1}\}$ with $|\mu(A_i) - p_i| < \epsilon$ for every $0 \leq i < |\pv|$.
\end{thm*}

Grillenberger and Krengel \cite{GK76} obtained a further strengthening of these results which roughly says that, under the assumptions of Denker's theorem, one can control the joint distribution of $\alpha$ and finitely many of its translates. In particular, they showed that under the assumptions of Denker's theorem there is a generating partition $\alpha$ with $\mu(A_i) = p_i$ for every $0 \leq i < |\pv|$.

Over the years, Krieger's theorem acquired much fame and underwent various generalizations. In 1972, Katznelson and Weiss \cite{KaW72} outlined a proof of Krieger's theorem for free ergodic actions of $\Z^d$. Roughly a decade later, \v{S}ujan \cite{Su83} stated Krieger's theorem for amenable groups but only outlined the proof. The first proof for amenable groups to appear in the literature was obtained in 1988 by Rosenthal \cite{Ros88} who proved Krieger's theorem under the more restrictive assumption that $\ksh_G(X, \mu) < \log(k-2) < \log(k)$. This was not improved until 2002 when Danilenko and Park \cite{DP02} proved Krieger's theorem for amenable groups under the assumption $\ksh_G(X, \mu) < \log(k-1) < \log(k)$. It is none-the-less a folklore unpublished result that Krieger's theorem holds for amenable groups, i.e. if $G \acts (X, \mu)$ is a free ergodic {\pmp} action of an amenable group and $\ksh_G(X, \mu) < \log(k)$ then there is a generating partition consisting of $k$ sets. Our much more general investigations here yield this as a consequence. We believe that this is the first explicit proof of this fact.

Krieger's theorem was also generalized to a relative setting. The relative version of Krieger's theorem for $\Z$ actions was first proven by Kifer and Weiss \cite{KiW02} in 2002. It states that if $\Z \acts (X, \mu)$ is a free ergodic {\pmp} action, $\cF$ is a $\Z$-invariant sub-$\sigma$-algebra, and the relative entropy satisfies $\ksh_\Z(X, \mu | \cF) < \log(k)$, then there is a Borel partition $\alpha$ consisting of $k$ sets such that $\salg_\Z(\alpha) \vee \cF = \mathcal{B}(X)$. This result was later extended by Danilenko and Park \cite{DP02} to free ergodic actions of amenable groups under the assumption that $\cF$ induces a class-bijective factor.

Rokhlin's theorem was generalized to actions of abelian groups by Conze \cite{C72} in 1972 and was just recently extended to amenable groups by Seward and Tucker-Drob \cite{ST14}. Specifically, if $G \acts (X, \mu)$ is a free ergodic {\pmp} action of an amenable group then the entropy $\ksh_G(X, \mu)$ is equal to the infimum of $\sH(\alpha)$ over all countable generating partitions $\alpha$. Denker's theorem on the other hand has not been extended beyond actions of $\Z$. 

In this paper we consider arbitrary countable groups, but we are particularly interested in the case of non-amenable groups. This is due to the recent breathtaking development of an entropy theory for actions of certain non-amenable groups. Specifically, groundbreaking work of Bowen in 2008 \cite{B10b}, followed with improvements by Kerr and Li \cite{KL11a}, has created the notion of \emph{sofic entropy} for {\pmp} actions of sofic groups. We remind the reader that the class of sofic groups contains the countable amenable groups, and it is an open question if every countable group is sofic. Sofic entropy in fact extends Kolmogorov--Sinai entropy as the two notions coincide for actions of amenable groups \cite{B12,KL13}. The sofic entropy of a Bernoulli shift $(L^G, \lambda^G)$ is equal to the Shannon entropy of its base $\sH(\lambda)$ \cite{B10b,KL11b}. Consequently, for sofic groups containing an infinite amenable subgroup, sofic entropy classifies Bernoulli shifts up to isomorphism \cite{Or70a,Or70b,OW87,St75}. Bowen has also made significant progress on classifying Bernoulli shifts over general countable groups \cite{B12b}, but a full classification does not yet exist.

This paper is motivated by some of the challenges facing sofic entropy theory. For instance, is Sinai's factor theorem true: do free ergodic actions of positive sofic entropy factor onto Bernoulli shifts? Is the Ornstein isomorphism theorem true: are Bernoulli shifts over sofic groups classified up to isomorphism by their sofic entropy? For both questions, the classical proofs involving elaborate constructions of partitions cannot be carried out. This is because vital properties of actions of amenable groups such as the Rokhlin lemma, the Shannon--McMillan--Breiman theorem, and the monotone decreasing property of entropy under factor maps are simply false for non-amenable groups. Additionally, the formula for sofic entropy involves counting auxiliary objects which are external to the original action, making it unclear if the language of sofic entropy is sufficiently rich to uncover the delicate constructions of partitions which are needed for answering these questions. We also do not know if sofic entropy satisfies either the Rokhlin generator theorem or the Krieger finite generator theorem. These are pertinent questions since sofic entropy is easier to define, compute, and understand when there exists a finite generating partition. Finally, a significant challenge to sofic entropy is that it is restricted to the realm of actions of sofic groups (in fact, to the realm of ``sofic actions''). If countable, non-sofic groups exist, how will we understand and classify their Bernoulli shifts?

Motivated by these issues, we introduce a new notion of entropy which is defined for all {\pmp} actions of all countable groups. For a {\pmp} action $G \acts (X, \mu)$ we define the \emph{Rokhlin entropy} to be
$$\rh_G(X, \mu) = \inf \Big\{ \sH(\alpha | \cJ) : \alpha \text{ is a countable partition and } \salg_G(\alpha) \vee \cJ = \mathcal{B}(X) \Big\},$$
where $\cJ$ is the $\sigma$-algebra of $G$-invariant Borel sets. We study this invariant in this three-part series, but in the present paper we only consider ergodic actions, in which case the Rokhlin entropy simplifies to
$$\rh_G(X, \mu) = \inf \Big\{ \sH(\alpha) \: \alpha \text{ is a countable Borel generating partition} \Big\}.$$
We name this invariant in honor of Rokhlin's generator theorem. For free actions of amenable groups, Rokhlin entropy coincides with Kolmogorov--Sinai entropy \cite{AS,ST14}. Thus Rokhlin entropy is a simple and natural extension of Kolmogorov--Sinai entropy.

More generally, if $\cF \subseteq \Borel(X)$ is a $G$-invariant $\sigma$-algebra, then we define the \emph{Rokhlin entropy of $G \acts (X, \mu)$ relative to $\cF$}, denoted $\rh_G(X, \mu | \cF)$, to be
$$\inf \Big\{ \sH(\alpha | \cF \vee \cJ) \: \alpha \text{ is a countable Borel partition and } \salg_G(\alpha) \vee \cF \vee \cJ = \mathcal{B}(X) \Big\}.$$
Again, $\cJ$ will always be trivial in this paper because we will only consider ergodic actions. We refer the reader to Section \ref{SECT PRELIM} for the definition of the conditional Shannon entropy $\sH(\alpha | \cF)$, but we remark that when $\cF = \{X, \varnothing\}$ we have $\sH(\alpha | \cF) = \sH(\alpha)$. In particular $\rh_G(X, \mu | \{X, \varnothing\}) = \rh_G(X, \mu)$.

We show in Proposition \ref{PROP RELROK} that for free ergodic actions of amenable groups relative Rokhlin entropy coincides with relative Kolmogorov--Sinai entropy (this is extended to non-ergodic actions in Part III \cite{AS}). We also observe in Proposition \ref{PROP OE} that $\rh_G(X, \mu | \cF)$ is invariant under orbit equivalences for which the orbit-change cocycle is $\cF$-measurable, generalizing a similar property of Kolmogorov--Sinai entropy discovered by Rudolph and Weiss \cite{RW00}.

Our main theorem is the following finite generator theorem which applies to all ergodic actions of countably infinite groups. (We in fact prove a stronger result; see Theorem \ref{INTRO THM2}).

\begin{thm} \label{INTRO THM1}
Let $G$ be a countably infinite group acting ergodically, but not necessarily freely, by measure-preserving bijections on a non-atomic standard probability space $(X, \mu)$. Let $\cF$ be a $G$-invariant sub-$\sigma$-algebra. If $\pv = (p_i)$ is any finite or countable probability vector with $\rh_G(X, \mu | \cF) < \sH(\pv)$, then there is a Borel partition $\alpha = \{A_i \: 0 \leq i < |\pv|\}$ with $\mu(A_i) = p_i$ for every $0 \leq i < |\pv|$ and with $\salg_G(\alpha) \vee \cF = \Borel(X)$.
\end{thm}

This theorem greatly supersedes previous work of the author in \cite{S12} which, under the assumption $\rh_G(X, \mu) < \infty$, constructed a finite generating partition without any control over its cardinality or distribution. The major difficulty which the present work overcomes is that all prior arguments for controlling the cardinality and distribution of generating partitions rely critically upon the classical Rokhlin lemma and Shannon--McMillan--Breiman theorem, and these tools do not exist for actions of general countable groups.

We remark that in order for a partition $\alpha$ to exist as described in Theorem \ref{INTRO THM1}, it is necessary that $\rh_G(X, \mu | \cF) \leq \sH(\pv)$. So the above theorem is optimal since in general there are actions where the infimum $\rh_G(X, \mu)$ is not achieved, such as free ergodic actions which are not isomorphic to any Bernoulli shift \cite{S14a}.

If $\rh_G(X, \mu) < \log(k)$ then using $\pv = (p_0, \ldots, p_{k-1})$ where each $p_i = 1 / k$ we obtain the following generalization of the (relative) Krieger finite generator theorem:

\begin{cor} \label{INTRO CORK}
Let $G$ be a countably infinite group acting ergodically, but not necessarily freely, by measure-preserving bijections on a non-atomic standard probability space $(X, \mu)$, and let $\cF$ be a $G$-invariant sub-$\sigma$-algebra. If $\rh_G(X, \mu | \cF) < \log(k)$, then there is a partition $\alpha$ with $|\alpha| = k$ and $\salg_G(\alpha) \vee \cF = \Borel(X)$.
\end{cor}

We mention that Corollary \ref{INTRO CORK} is the first version of Krieger's finite generator theorem for non-free actions. Furthermore, we believe that Corollary \ref{INTRO CORK} (together with the Rokhlin generator theorem for amenable groups \cite{ST14}) is the first explicit proof of Krieger's finite generator theorem for free ergodic actions of countable amenable groups. In fact, we obtain the following strong form of Denker's theorem for amenable groups:

\begin{cor}
Let $G$ be a countably infinite amenable group and let $G \acts (X, \mu)$ be a free ergodic {\pmp} action. If $\pv = (p_i)$ is any finite or countable probability vector with $\ksh_G(X, \mu) < \sH(\pv)$ then there exists a generating partition $\alpha = \{A_i \: 0 \leq i < |\pv|\}$ with $\mu(A_i) = p_i$ for every $0 \leq i < |\pv|$.
\end{cor}

Returning to our points of motivation, we point out that Theorem \ref{INTRO THM1} implies that sofic entropy will satisfy the Krieger finite generator theorem provided it satisfies the Rokhlin generator theorem (i.e. provided sofic entropy coincides with Rokhlin entropy). A simple consequence of the definitions is that sofic entropy is bounded above by Rokhlin entropy \cite{AS, B10b}, and its an important open problem to determine if they are equal (assuming a free action with sofic entropy not minus infinity).

Since sofic entropy is a lower bound to Rokhlin entropy, we know that Bernoulli shifts $(L^G, \lambda^G)$ over sofic groups $G$ have Rokhlin entropy $\rh_G(L^G, \lambda^G) = \sH(\lambda)$. Since the definition of Rokhlin entropy does not mention soficity, this suggests the equality $\rh_G(L^G, \lambda^G) = \sH(\lambda)$ may hold for all countably infinite groups. If so, Rokhlin entropy could be capable of classifying Bernoulli shifts over all countably infinite groups up to isomorphism (in particular, over non-sofic groups if they exist). We further investigate this question in Part II \cite{S14a}.

In the long-term, we hope that Rokhlin entropy will not only be useful in its own right, but that its study will develop in parallel with sofic entropy theory and that the two theories will be mutually beneficial to one another. In particular, we hope Rokhlin entropy will enrich the language and framework available for studying entropy-type problems and that it will progress our techniques for constructing partitions, possibly leading the way to generalizations of the many deep results of Ornstein.

A hidden significance of Theorem \ref{INTRO THM1} is that it opens the door to developing a theory of Rokhlin entropy. It should be pointed out that the definition of Rokhlin entropy is both quite natural and immediately suggested by Rokhlin's generator theorem, and the idea of its definition had certainly occurred to researchers beforehand. However, the abstract nature of the definition, an infimum over {\it all} generating partitions, seems to prevent any viable means of study. Our main theorem changes this situation. It reveals, as a consequence, a sub-additive property of Rokhlin entropy. To properly state this property in its strongest form requires additional definitions and is postponed to Part II \cite{S14a}, but we mention here a simple corollary to give an indication.

\begin{cor} \label{INTRO COR1}
Let $G$ be a countably infinite group acting ergodically, but not necessarily freely, by measure-preserving bijections on a non-atomic standard probability space $(X, \mu)$. If $G \acts (Y, \nu)$ is a factor of $G \acts (X, \mu)$ and $\cF$ is the sub-$\sigma$-algebra of $X$ associated to $Y$ then
$$\rh_G(X, \mu) \leq \rh_G(Y, \nu) + \rh_G(X, \mu | \cF).$$
\end{cor}

For example, if $\alpha$ and $\beta$ are partitions with $\salg_G(\alpha \vee \beta) = \Borel(X)$, then the above corollary implies $\rh_G(X, \mu) \leq \sH(\alpha) + \sH(\beta | \salg_G(\alpha))$. The inequality in Corollary \ref{INTRO COR1} can be strict, such as when $\rh_G(X, \mu) < \rh_G(Y, \nu)$. A strict inequality is common for actions of non-amenable groups \cite{S13}.

The sub-additive property turns out to be tremendously useful, and it is absolutely critical to our study of Rokhlin entropy in Parts II and III.

\subsection*{Update}
Between this article's first appearance on the arXiv and it reaching its final form, a few developments have occurred.

Rokhlin entropy theory was ultimately successful in providing a framework for generalizing the Sinai factor theorem to all countably infinite groups. Specifically, if $G \acts (X, \mu)$ is a free ergodic action of any countably infinite group $G$, then $G \acts (X, \mu)$ factors onto the Bernoulli shift $G \acts (L^G, \lambda^G)$ whenever $\rh_G(X, \mu) \geq \sH(\lambda)$ \cite{S16a}. In particular, actions which do not admit any finite generating partitions must factor onto all Bernoulli shifts. Also, since sofic entropy is bounded above by Rokhlin entropy, this implies that sofic entropy satisfies the Sinai factor theorem as well.

In \cite{S16} two fairly computable expressions were found which provide upper bounds to Rokhlin entropy. In particular, it was shown that if $G \acts (X, \mu)$ is a free {\pmp} action and $\alpha$ is a generating partition with $\sH(\alpha) < \infty$, then the Rokhlin entropy satisfies
$$\rh_G(X, \mu) \leq \inf_{T \subseteq G} \frac{1}{|T|} \cdot \sH( \textstyle{\bigvee_{t \in T}} t \cdot \alpha),$$
where the infimum is over all finite $T \subseteq G$. When $G$ is amenable, the right-hand side coincides with Kolmogorov--Sinai entropy (it is common to use a sequence of F{\o}lner sets $T$, but this isn't necessary \cite{OW87}). Thus, in some sense this upper bound should be no more difficult to compute in practice than Kolmogorov--Sinai entropy. We point out that this upper bound makes our finite generator theorem, Theorem \ref{INTRO THM1}, easier to apply.

In addition to Parts II and III \cite{S14a,AS} and the papers \cite{S16, S16a} mentioned above, additional study of Rokhlin entropy has been undertaken in \cite{Al, B16, GS15}. The sub-additive property of Rokhlin entropy continues to serve as the foundation for all of these new results.

\subsection*{Outline}
The proof of Theorem \ref{INTRO THM1} is essentially self-contained as it relies on technical constructions carried out by hand. The proof makes significant use of the pseudo-group of the induced orbit-equivalence relation. We review basic properties of the pseudo-group in  Section \ref{SECT PSEUDO}. In Section \ref{SECT INFT} we review and strengthen a construction of the author used in \cite{S12} for replacing countably infinite partitions with finite ones. Sections \ref{SECT PSEUDO} and \ref{SECT INFT} thus reprove the main theorem of \cite{S12} which states that finite Rokhlin entropy implies the existence of a finite generating partition.

The real difficulty of the present work is constructing a generating partition while controlling its cardinality and distribution. The classical Rokhlin lemma and Shannon--McMillan--Breiman theorem were critical to this task in all prior proofs of Krieger's theorem. The important advantage we obtain by working with the pseudo-group is that we are able to develop a replacement to the Rokhlin lemma and the Shannon--McMillan--Breiman theorem which is suitable to our needs. We present this replacement in Section \ref{SECT EQREL}.

A new, significant difficulty is created from our use of the pseudo-group. Specifically, the notion of a ``generating'' partition cannot be expressed in the language of the pseudo-group. Ultimately, we must build a single, efficient partition which both codes information for certain transformations in the pseudo-group, and simultaneously codes information for the action of $G$. An obstacle in this simultaneous-coding problem is that there is no geometric relationship between our pseudo-group transformations and the $G$-action. This is the most challenging part of the proof. The coding machinery needed for this task is presented in Section \ref{SECT ACT}. Then in Section \ref{SECT KRIEGER} we prove the main theorem. Finally, in Section \ref{SECT AMENABLE} we show that relative Rokhlin entropy and relative Kolmogorov--Sinai entropy coincide for free actions of amenable groups.

\subsection*{Acknowledgments}
This research was supported by the National Science Foundation Graduate Student Research Fellowship under Grant No. DGE 0718128. The author thanks his advisor, Ralf Spatzier, for numerous helpful discussions, Tim Austin for many suggestions to improve the paper, and Miklos Ab\'{e}rt and Benjy Weiss for encouraging the author to coin a name for the new invariant studied here. Part of this work was completed while the author attended the Arbeitsgemeinschaft: Sofic Entropy workshop at the Mathematisches Forschungsinstitut Oberwolfach in Germany. The author thanks the MFO for their hospitality and travel support.

\section{Preliminaries} \label{SECT PRELIM}

Every probability space $(X, \mu)$ which we consider will be assumed to be standard. In particular, $X$ will be a standard Borel space. For $\xi \subseteq \mathcal{B}(X)$, we let $\salg(\xi)$ denote the smallest sub-$\sigma$-algebra containing $\xi$ (not to be confused with the notation $\salg_G(\xi)$ from the introduction). At times, we will consider the space of all Borel probability measures on $X$. Recall that the space of Borel probability measures on $X$ has a natural standard Borel structure which is generated by the maps $\lambda \mapsto \lambda(A)$ for $A \subseteq X$ Borel \cite[Theorem 17.24]{K95}.

An action $G \acts (Y, \nu)$ is a \emph{factor} of $G \acts (X, \mu)$ if there exists a measure-preserving $G$-equivariant map $\pi : (X, \mu) \rightarrow (Y, \nu)$. Every factor $\pi : (X, \mu) \rightarrow (Y, \nu)$ is uniquely associated (mod $\mu$-null sets) to a $G$-invariant sub-$\sigma$-algebra $\cF$ of $X$, and conversely every $G$-invariant sub-$\sigma$-algebra $\cF$ of $(X, \mu)$ is uniquely associated (up to isomorphism) to a factor $\pi : (X, \mu) \rightarrow (Y, \nu)$ \cite[Theorem 2.15]{Gl03}. If $\pi : (X, \mu) \rightarrow (Y, \nu)$ is a factor map, then there is an essentially unique Borel map associating each $y \in Y$ to a Borel probability measure $\mu_y$ on $X$ such that $\mu = \int \mu_y \ d \nu(y)$ and $\mu_y(\pi^{-1}(y)) = 1$. We call this the \emph{disintegration} of $\mu$ over $\nu$. Note that for any Borel set $A \subseteq X$, the map $y \mapsto \mu_y(A)$ is Borel.

Let $G \acts (X, \mu)$ be a {\pmp} action, and let $\cF$ be a $G$-invariant sub-$\sigma$-algebra. Let $\pi : (X, \mu) \rightarrow (Y, \nu)$ be the associated factor, and let $\mu = \int \mu_y \ d \nu(y)$ be the disintegration of $\mu$ over $\nu$. For a countable Borel partition $\alpha$ of $X$, the \emph{conditional Shannon entropy} of $\alpha$ relative to $\cF$ is
$$\sH(\alpha | \cF) = \int_Y \sum_{A \in \alpha} - \mu_y(A) \cdot \log(\mu_y(A)) \ d \nu(y) = \int_Y \sH_{\mu_y}(\alpha) \ d \nu(y).$$
If $\cF = \{X, \varnothing\}$ is the trivial $\sigma$-algebra then $\sH(\alpha | \cF) = \sH(\alpha)$. For a countable partition $\beta$ of $X$ we set $\sH(\alpha | \beta) = \sH(\alpha | \salg(\beta))$. We write $\alpha \geq \beta$ if $\alpha$ is finer than $\beta$. We will need the following standard properties of Shannon entropy (proofs can be found in \cite{Do11}):

\begin{lem} \label{LEM SHAN}
Let $(X, \mu)$ be a standard probability space, let $\alpha$ and $\beta$ be countable Borel partitions of $X$, and let $\cF$ and $\Sigma$ be sub-$\sigma$-algebras. Then
\begin{enumerate}
\item[\rm (i)] $\sH(\alpha | \cF) \leq \log |\alpha|$;
\item[\rm (ii)] if $\alpha \geq \beta$ then $\sH(\alpha | \cF) \geq \sH(\beta | \cF)$;
\item[\rm (iii)] if $\Sigma \subseteq \cF$ then $\sH(\alpha | \Sigma) \geq \sH(\alpha | \cF)$;
\item[\rm (iv)] $\sH(\alpha \vee \beta) = \sH(\beta) + \sH(\alpha | \beta) \geq \sH(\beta)$;
\item[\rm (v)] $\sH(\alpha \vee \beta | \cF) = \sH(\beta | \cF) + \sH(\alpha | \salg(\beta) \vee \cF) \leq \sH(\beta | \cF) + \sH(\alpha)$;
\item[\rm (vi)] $\sH(\alpha | \cF) = \sup_\xi \sH(\xi | \cF)$, where the supremum is over all finite partitions $\xi$ coarser than $\alpha$;
\item[\rm (vii)] if $\sH(\alpha) < \infty$ then $\sH(\alpha | \cF) = \inf_\xi \sH(\alpha | \xi)$, where the infimum is over all finite partitions $\xi \subseteq \cF$.
\end{enumerate}
\end{lem}

Throughout this paper, whenever working with a probability space $(X, \mu)$ we will generally ignore sets of measure zero. In particular, we write $A = B$ for $A, B \subseteq X$ if their symmetric difference is null. We similarly write $\cF = \Sigma$ for sub-$\sigma$-algebras $\cF$, $\Sigma$ if they agree up to null sets. Also, by a partition of $X$ we will mean a countable collection of pairwise-disjoint Borel sets whose union is conull. In particular, we allow partitions to contain the empty set. Similarly, we will use the term \emph{probability vector} more freely than described in the introduction. A probability vector $\pv = (p_i)$ will be any finite or countable ordered tuple of non-negative real numbers which sum to $1$ (so some terms $p_i$ may be $0$). We say that another probability vector $\qv$ is \emph{coarser} than $\pv$ if there is a partition $\cQ = \{Q_j \: 0 \leq j < |\qv|\}$ of the integers $\{0 \leq i < |\pv|\}$ such that for every $0 \leq j < |\qv|$
$$q_j = \sum_{i \in Q_j} p_i.$$

A \emph{pre-partition} of $X$ is a countable collection of pairwise-disjoint subsets of $X$. We say that another pre-partition $\beta$ extends $\alpha$, written $\beta \sqsupseteq \alpha$, if there is an injection $\iota : \alpha \rightarrow \beta$ with $A \subseteq \iota(A)$ for every $A \in \alpha$. Equivalently, $\beta \sqsupseteq \alpha$ if and only if $\cup \alpha \subseteq \cup \beta$ and the restriction of $\beta$ to $\cup \alpha$ coincides with $\alpha$.

For a Borel pre-partition $\alpha$, we define the \emph{reduced $\sigma$-algebra} $\ralg_G(\alpha)$ to be the collection of Borel sets $R \subseteq X$ such that there is a conull $X' \subseteq X$ satisfying:
\begin{quote}
for every $r \in R \cap X'$ and $x \in X' \setminus R$ there is $g \in G$ with $g \cdot r, g \cdot x \in \cup \alpha$ and with $g \cdot r$ and $g \cdot x$ lying in distinct classes of $\alpha$.
\end{quote}
It is a basic exercise to verify that $\ralg_G(\alpha)$ is indeed a $\sigma$-algebra. We note the following basic property.

\begin{lem} \label{LEM EXT}
Let $G \acts (X, \mu)$ be a {\pmp} action, and let $\alpha$ be a pre-partition. If $\beta$ is a pre-partition and $\beta \sqsupseteq \alpha$ then $\ralg_G(\beta) \supseteq \ralg_G(\alpha)$. In particular, if $\beta$ is a partition and $\beta \sqsupseteq \alpha$ then $\salg_G(\beta) \supseteq \ralg_G(\alpha)$.
\end{lem}

\begin{proof}
Fix $R \in \ralg_G(\alpha)$. By definition of $\ralg_G(\alpha)$, there is a conull $X' \subseteq X$ such that for all $r \in R \cap X'$ and $x \in X' \setminus R$ there is $g \in G$ with $g \cdot r, g \cdot x \in \cup \alpha$ and such that $\alpha$ separates $g \cdot r$ and $g \cdot x$. Since the restriction of $\beta$ to $\cup \alpha$ is equal to $\alpha$, we also have that $g \cdot r, g \cdot x \in \cup \beta$ and $\beta$ separates $g \cdot r$ and $g \cdot x$. We conclude that $R \in \ralg_G(\beta)$.
\end{proof}

The definition of reduced $\sigma$-algebra may seem a bit odd at first, but comes about naturally from our work here and will significantly simplify some of the proofs in Part II and Part III \cite{S14a, AS}. A key property of this definition is that if $\beta$ is any partition extending $\alpha$ then one automatically has $\salg_G(\beta) \supseteq \ralg_G(\alpha)$. Another important property is that if $G \acts (Y, \nu)$ is a factor of $(X, \mu)$ via $\phi : (X, \mu) \rightarrow (Y, \nu)$, then for any pre-partition $\alpha$ of $Y$ we have $\phi^{-1}(\ralg_G(\alpha)) \subseteq \ralg_G(\phi^{-1}(\alpha))$. These properties can be quite useful for specialized constructions. For example, one could imagine constructing two pre-partitions $\alpha^1$ and $\alpha^2$ which achieve different goals. If $\cup \alpha^1$ is disjoint from $\cup \alpha^2$, then one can choose a common extension partition $\alpha$ and automatically have $\salg_G(\alpha) \supseteq \ralg_G(\alpha^1) \vee \ralg_G(\alpha^2)$. This type of construction will be performed in Part II.

Below is the statement of the main theorem of this paper in its strongest form. It is a strengthening of Theorem \ref{INTRO THM1} mentioned in the introduction. This theorem is new even in the case $G = \Z$.

\begin{thm} \label{INTRO THM2}
Let $G$ be a countably infinite group acting ergodically, but not necessarily freely, by measure-preserving bijections on a non-atomic standard probability space $(X, \mu)$. Let $\cF$ be a $G$-invariant sub-$\sigma$-algebra of $X$. If $0 < r \leq 1$ and $\pv = (p_i)$ is any finite or countable probability vector with $\rh_G(X, \mu | \cF) < r \cdot \sH(\pv)$, then there is a Borel pre-partition $\alpha = \{A_i \: 0 \leq i < |\pv|\}$ with $\mu(\cup \alpha) = r$, $\mu(A_i) = r \cdot p_i$ for every $0 \leq i < |\pv|$, and $\ralg_G(\alpha) \vee \cF = \mathcal{B}(X)$.
\end{thm}

\section{The pseudo-group of an ergodic action} \label{SECT PSEUDO}

For a {\pmp} action $G \acts (X, \mu)$ we let $E_G^X$ denote the induced orbit equivalence relation:
$$E_G^X = \{(x, y) \: \exists g \in G, \ \ g \cdot x = y\}.$$
The \emph{pseudo-group} of $E_G^X$, denoted $[[E_G^X]]$, is the set of all Borel bijections $\theta : \dom(\theta) \rightarrow \rng(\theta)$ where $\dom(\theta), \rng(\theta) \subseteq X$ are Borel and $\theta(x) \in G \cdot x$ for every $x \in \dom(\theta)$. The \emph{full group} of $E_G^X$, denoted $[E_G^X]$, is the set of all $\theta \in [[E_G^X]]$ with $\dom(\theta) = \rng(\theta) = X$ (i.e. conull in $X$).

For every $\theta \in [[E_G^X]]$ there is a Borel partition $\{Z_g^\theta \: g \in G\}$ of $\dom(\theta)$ such that $\theta(x) = g \cdot x$ for every $x \in Z_g^\theta$. Thus, an important fact which we will use repeatedly is that every $\theta \in [[E_G^X]]$ is measure-preserving. We mention that the sets $Z_g^\theta$ are in general not uniquely determined from $\theta$ since the action of $G$ might not be free. It will be necessary to keep record of such decompositions $\{Z_g^\theta\}$ for $\theta \in [[E_G^X]]$. The precise notion we need is the following.

\begin{defn}
Let $G \acts (X, \mu)$ be a {\pmp} action, let $\theta \in [[E_G^X]]$, and let $\cF$ be a $G$-invariant sub-$\sigma$-algebra. We say that $\theta$ is \emph{$\cF$-expressible} if $\dom(\theta), \rng(\theta) \in \cF$ and there is a $\cF$-measurable partition $\{Z_g^\theta \: g \in G\}$ of $\dom(\theta)$ such that $\theta(x) = g \cdot x$ for every $x \in Z_g^\theta$ and all $g \in G$.
\end{defn}

We observe two simple facts on the notion of expressibility.

\begin{lem} \label{LEM EXPMOVE}
Let $G \acts (X, \mu)$ be a {\pmp} action and let $\cF$ be a $G$-invariant sub-$\sigma$-algebra. If $\theta \in [[E_G^X]]$ is $\cF$-expressible and $A \subseteq X$, then $\theta(A) = \theta(A \cap \dom(\theta))$ is $\salg_G(\{A\}) \vee \cF$-measurable. In particular, if $A \in \cF$ then $\theta(A) \in \cF$.
\end{lem}

\begin{proof}
Fix a $\cF$-measurable partition $\{Z_g^\theta \: g \in G\}$ of $\dom(\theta)$ such that $\theta(x) = g \cdot x$ for all $x \in Z_g^\theta$. Then
\begin{equation*}
\theta(A) = \bigcup_{g \in G} g \cdot (A \cap Z_g^\theta) \in \salg_G(\{A\}) \vee \cF. \qedhere
\end{equation*}
\end{proof}

\begin{lem} \label{LEM EXPGROUP}
Let $G \acts (X, \mu)$ be a {\pmp} action and let $\cF$ be a $G$-invariant sub-$\sigma$-algebra. If $\theta, \phi \in [[E_G^X]]$ are $\cF$-expressible then so are $\theta^{-1}$ and $\theta \circ \phi$.
\end{lem}

\begin{proof}
Fix $\cF$-measurable partitions $\{Z_g^\theta \: g \in G\}$ and $\{Z_g^\phi \: g \in G\}$ of $\dom(\theta)$ and $\dom(\phi)$, respectively, satisfying $\theta(x) = g \cdot x$ for all $x \in Z_g^\theta$ and $\phi(x) = g \cdot x$ for all $x \in Z_g^\phi$. Define for $g \in G$
$$Z_g^{\theta^{-1}} = g^{-1} \cdot Z_{g^{-1}}^\theta.$$
Then each $Z_g^{\theta^{-1}}$ is $\cF$-measurable since $\cF$ is $G$-invariant. It is easily checked that $\{Z_g^{\theta^{-1}} \: g \in G\}$ partitions $\rng(\theta)$ and satisfies $\theta^{-1}(x) = g \cdot x$ for all $x \in Z_g^{\theta^{-1}}$. Thus $\theta^{-1}$ is $\cF$-expressible.

Observe that by the previous lemma, $\phi^{-1}(Z_g^\theta) \in \cF$ for every $g \in G$ since $\phi^{-1}$ is $\cF$-expressible. Notice that the sets $Z_g^\phi \cap \phi^{-1}(Z_h^\theta)$ partition $\dom(\theta \circ \phi)$. Define for $g \in G$
$$Z_g^{\theta \circ \phi} = \bigcup_{h \in G} \Big( Z_{h^{-1} g}^\phi \cap \phi^{-1}(Z_h^\theta) \Big).$$
These sets are $\cF$-measurable and pairwise-disjoint and we have $\theta \circ \phi(x) = g \cdot x$ for all $x \in Z_g^{\theta \circ \phi}$.
\end{proof}

With the aid of Lemma \ref{LEM EXPMOVE}, we observe a basic property of relative Rokhlin entropy. The proposition below resembles a theorem of Rudolph and Weiss from classical entropy theory \cite{RW00}. Note that if $G$ and $\Gamma$ act on $(X, \mu)$ with the same orbits then $E_G^X = E_\Gamma^X$ and $[[E_G^X]] = [[E_\Gamma^X]]$. In this situation, we say that $\theta \in [[E_G^X]]$ is $(G, \cF)$-expressible if it is $\cF$-expressible with respect to the $G$-action $G \acts (X, \mu)$.

\begin{prop} \label{PROP OE}
Let $G$ and $\Gamma$ be countable groups, and let $G \acts (X, \mu)$ and $\Gamma \acts (X, \mu)$ be {\pmp} ergodic actions having the same orbits. Suppose that $\cF$ is a $G$ and $\Gamma$ invariant sub-$\sigma$-algebra such that the transformation associated to each $g \in G$ is $(\Gamma, \cF)$-expressible and similarly the transformation associated to each $\gamma \in \Gamma$ is $(G, \cF)$-expressible. Then
$$\rh_G(X, \mu | \cF) = \rh_\Gamma(X, \mu | \cF).$$
\end{prop}

\begin{proof}
It suffices to show that for every countable partition $\alpha$, $\salg_G(\alpha) \vee \cF = \salg_\Gamma(\alpha) \vee \cF$. Indeed, since the transformation associated to each $g \in G$ is $(\Gamma, \cF)$-expressible, it follows from Lemma \ref{LEM EXPMOVE} that the $\sigma$-algebra $\salg_\Gamma(\alpha) \vee \cF$ is $G$-invariant and contains $\alpha$. Therefore $\salg_G(\alpha) \vee \cF \subseteq \salg_\Gamma(\alpha) \vee \cF$. With the same argument we obtain the reverse containment.
\end{proof}

The lemma below and the corollaries which follow it provide us with all elements of the pseudo-group $[[E_G^X]]$ which will be needed in forthcoming sections.

\begin{lem} \label{LEM SIMPLEMIX}
Let $G \acts (X, \mu)$ be an ergodic {\pmp} action. Let $A, B \subseteq X$ be Borel sets with $0 < \mu(A) \leq \mu(B)$. Then there exists a $\salg_G(\{A, B\})$-expressible function $\theta \in [[E_G^X]]$ with $\dom(\theta) = A$ and $\rng(\theta) \subseteq B$.
\end{lem}

\begin{proof}
Let $g_0, g_1, \ldots$ be an enumeration of $G$. Set $Z_{g_0}^\theta = A \cap g_0^{-1} \cdot B$ and inductively define
$$Z_{g_n}^\theta = \Big( A \setminus \Big( \textstyle{\bigcup_{i = 0}^{n-1} Z_{g_i}^\theta} \Big) \Big) \bigcap g_n^{-1} \cdot \Big( B \setminus \Big( \textstyle{\bigcup_{i = 0}^{n - 1} g_i \cdot Z_{g_i}^\theta} \Big) \Big).$$
Define $\theta : \bigcup_{n \in \N} Z_{g_n}^\theta \rightarrow B$ by setting $\theta(x) = g_n \cdot x$ for $x \in Z_{g_n}^\theta$. Clearly $\theta$ is $\salg_G(\{A, B\})$-expressible.

Set $C = A \setminus \dom(\theta)$. Towards a contradiction, suppose that $\mu(C) > 0$. Then we have
$$\mu(\rng(\theta)) = \mu(\dom(\theta)) < \mu(A) \leq \mu(B).$$
So $\mu(B \setminus \rng(\theta)) > 0$ and by ergodicity there is $n \in \N$ with
$$\mu \Big( C \cap g_n^{-1} \cdot (B \setminus \rng(\theta)) \Big) > 0.$$
However, this implies that $\mu(C \cap Z_{g_n}^\theta) > 0$, a contradiction. We conclude that, up to a null set, $\dom(\theta) = A$.
\end{proof}

\begin{cor} \label{COR MAKEPART}
Let $G \acts (X, \mu)$ be a {\pmp} ergodic action. If $C \subseteq B \subseteq X$ and $\mu(C) = \frac{1}{n} \cdot \mu(B)$ with $n \in \N$, then there is a $\salg_G(\{C, B\})$-measurable partition $\xi$ of $B$ into $n$ pieces with each piece having measure $\frac{1}{n} \cdot \mu(B)$ and with $C \in \xi$.
\end{cor}

\begin{proof}
Set $C_1 = C$. Once $\salg_G(\{C, B\})$-measurable subsets $C_1, \ldots, C_{k-1}$ of $B$, each of measure $\frac{1}{n} \cdot \mu(B)$, have been defined, we apply Lemma \ref{LEM SIMPLEMIX} to get a $\salg_G(\{C, B\})$-expressible function $\theta \in [[E_G^X]]$ with $\dom(\theta) = C$ and
$$\rng(\theta) \subseteq B \setminus (C_1 \cup \cdots \cup C_{k-1}).$$
We set $C_k = \theta(C)$. We note that $\mu(C_k) = \frac{1}{n} \cdot \mu(B)$ and $C_k \in \salg_G(\{C, B\})$ by Lemma \ref{LEM EXPMOVE}. Finally, set $\xi = \{C_1, \ldots, C_n\}$.
\end{proof}

In the corollary below we write $\id_A \in [[E_G^X]]$ for the identity function on $A$ for $A \subseteq X$.

\begin{cor} \label{COR PERMUTE}
Let $G \acts (X, \mu)$ be an ergodic {\pmp} action. If $\xi = \{C_1, \ldots, C_n\}$ is a collection of pairwise disjoint Borel sets of equal measure, then there is a $\salg_G(\xi)$-expressible function $\theta \in [[E_G^X]]$ which cyclically permutes the members of $\xi$, meaning that $\dom(\theta) = \rng(\theta) = \cup \xi$, $\theta(C_k) = C_{k+1}$ for $1 \leq k < n$, $\theta(C_n) = C_1$, and $\theta^n = \id_{\cup \xi}$.
\end{cor}

\begin{proof}
By Lemma \ref{LEM SIMPLEMIX}, for each $2 \leq k \leq n$ there is a $\salg_G(\xi)$-expressible function $\phi_k \in [[E_G^X]]$ with $\dom(\phi_k) = C_1$ and $\rng(\phi_k) = C_k$. We define $\theta : \cup \xi \rightarrow \cup \xi$ by
$$\theta(x) = \begin{cases}
\phi_2(x) & \text{if } x \in C_1\\
\phi_{k+1} \circ \phi_k^{-1}(x) & \text{if } x \in C_k \text{ and } 1 < k < n\\
\phi_n^{-1}(x) & \text{if } x \in C_n.
\end{cases}$$
Then $\theta$ cyclically permutes the members of $\xi$ and has order $n$. Finally, each restriction $\theta \res C_k$ is $\salg_G(\xi)$-expressible by Lemma \ref{LEM EXPGROUP} and thus $\theta$ is $\salg_G(\xi)$-expressible.
\end{proof}

\section{Countably infinite partitions} \label{SECT INFT}

In this section, we show how to replace countably infinite partitions by finite ones. This will allow us to carry-out counting arguments in proving the main theorem. Our work in this section retraces and improves upon methods used by the author in \cite{S12}. We improve upon \cite{S12} in two ways. First, we work in a relative setting where a $G$-invariant sub-$\sigma$-algebra $\cF$ is given, and second we show that one can control the Shannon entropy of the newly constructed partition.

For a finite set $S$ we let $S^{< \omega}$ denote the set of all finite words with letters in $S$ (the $\omega$ in the superscript denotes the first infinite ordinal). For $z \in S^{< \omega}$ we let $|z|$ denote the length of the word $z$. The lemma below is a relativized version of a similar result due to Krieger \cite{Kr70}.

\begin{lem} \label{LEM KRIEGER}
Let $(X, \mu)$ be a probability space, let $\cF$ be a sub-$\sigma$-algebra, let $(Y, \nu)$ be the associated factor of $(X, \mu)$, and let $\mu = \int_Y \mu_y \ d \nu(y)$ be the disintegration of $\mu$ over $\nu$. If $\xi$ is a countable Borel partition of $X$ with $\sH(\xi | \cF) < \infty$, then there is a Borel function $L : Y \times \xi \rightarrow \{0, 1, 2\}^{<\omega}$ which has finite average length
$$\int_Y \sum_{C \in \xi} |L(y, C)| \cdot \mu_y(C) \ d \nu(y) < \infty$$
and such that $\nu$-almost-every restriction $L(y, \cdot) : \xi \rightarrow \{0, 1, 2\}^{< \omega}$ is essentially injective in the sense that $L(y, C) = L(y, C')$ and $C \neq C'$ implies $\mu_y(C) \cdot \mu_y(C') = 0$.
\end{lem}

\begin{proof}
If $\xi$ is finite then we can simply fix an injection $L : \xi \rightarrow \{0, 1, 2\}^k$ for some $k \in \N$. So suppose that $\xi$ is infinite. Say $\xi = \{C_1, C_2, \ldots\}$. For $y \in Y$ let $\sigma(y) : \N \rightarrow \N$ be the unique bijection satisfying for all $n \in \N$: either $\mu_y(C_{\sigma(y)(n+1)}) < \mu_y(C_{\sigma(y)(n)})$ or else $\mu_y(C_{\sigma(y)(n+1)}) = \mu_y(C_{\sigma(y)(n)})$ and $\sigma(y)(n+1) > \sigma(y)(n)$. Since each map $y \mapsto \mu_y(C_k)$ is Borel (see \S\ref{SECT PRELIM}), we see that $\sigma : Y \rightarrow \N^\N$ is Borel.

For each $n$ let $t(n) \in \{0, 1, 2\}^{<\omega}$ be the ternary expansion of $n$. Note that $|t(n)| \leq \log_3(n) + 1$. For $y \in Y$ define
$$L(y, C_{\sigma(y)(n)}) = t(n) \text{ for } n \in \N \quad \text{and} \quad L(y, C_k) = t(1) \text{ for } k \in \N \setminus \sigma(y)(\N).$$
If $|t(n)| = |L(y, C_{\sigma(y)(n)})| > - \log \mu_y(C_{\sigma(y)(n)})$ then for all $k \leq n$
$$\mu_y(C_{\sigma(y)(k)}) \geq \mu_y(C_{\sigma(y)(n)}) > e^{-|t(n)|} \geq \frac{1}{e} \cdot e^{- \log_3(n)} = \frac{1}{e} \cdot n^{- \log_3(e)}.$$
Thus
$$\frac{1}{e} \cdot n^{1 - \log_3(e)} = n \cdot \frac{1}{e} \cdot n^{- \log_3(e)} < \sum_{k = 1}^n \mu_y(C_{\sigma(y)(k)}) \leq 1,$$
and hence $n \leq \exp(1 / (1 - \log_3(e)))$. Letting $m$ be the least integer greater than $\exp(1 / (1 - \log_3(e)))$, we have that $|L(y, C_{\sigma(y)(n)})| \leq - \log \mu_y(C_{\sigma(y)(n)})$ for all $y \in Y$ and all $n > m$. Therefore, recalling that $\mu_y(C_k) = 0$ for all $k \in \N \setminus \sigma(y)(\N)$, we have
\begin{align*}
\sum_{n \in \N} |L(y, C_n)| \cdot \mu_y(C_n) & = \sum_{n \in \N} |L(y, C_{\sigma(y)(n)})| \cdot \mu_y(C_{\sigma(y)(n)})\\
 & \leq m \cdot |t(m)| + \sum_{n > m} |L(y, C_{\sigma(y)(n)})| \cdot \mu_y(C_{\sigma(y)(n)}) \\
 & \leq m \cdot |t(m)| + \sum_{n \in \N} - \mu_y(C_n) \log \mu_y(C_n) \\
 & = m \cdot |t(m)| + \sH_{\mu_y}(\xi).
\end{align*}
Integrating both sides over $Y$ and using $\int_Y \sH_{\mu_y}(\xi) \ d \nu(y) = \sH(\xi | \cF) < \infty$ completes the proof.
\end{proof}

\begin{prop} \label{PROP FINGEN}
Let $G \acts (X, \mu)$ be an ergodic {\pmp} action, let $\cF$ be a $G$-invariant sub-$\sigma$-algebra, and let $\xi$ be a countable Borel partition with $\sH(\xi | \cF) < \infty$. Then for every $\epsilon > 0$ there is a finite Borel partition $\alpha$ with $\salg_G(\alpha) \vee \cF = \salg_G(\xi) \vee \cF$ and $\sH(\alpha | \cF) < \sH(\xi | \cF) + \epsilon$.
\end{prop}

\begin{proof}
Let $\pi : (X, \mu) \rightarrow (Y, \nu)$ be the factor map associated to $\cF$, and let $\mu = \int \mu_y \ d \nu(y)$ be the disintegration of $\mu$ over $\nu$. Apply Lemma \ref{LEM KRIEGER} to obtain a Borel function $L: Y \times \xi \rightarrow \{0, 1, 2\}^{< \omega}$ such that $\nu$-almost-every restriction $L(y, \cdot) : \xi \rightarrow \{0, 1, 2\}^{<\omega}$ is essentially injective and
$$\int_Y \sum_{C \in \xi} |L(y, C)| \cdot \mu_y(C) \ d \nu(y) < \infty.$$
We define $\ell : X \rightarrow \{0, 1, 2\}^{< \omega}$ by
$$\ell(x) = L(\pi(x), C)$$
for $x \in C \in \xi$. Observe that $\ell$ is $\salg(\xi) \vee \cF$-measurable and
$$\int_X |\ell(x)| \ d \mu(x) = \int_Y \int_X |\ell(x)| \ d \mu_y(x) \ d \nu(y) = \int_Y \sum_{C \in \xi} |L(y, C)| \cdot \mu_y(C) \ d \nu(y) < \infty.$$

For $n \in \N$ let $\cP_n = \{P_n, X \setminus P_n\}$ where
$$P_n = \{x \in X \: |\ell(x)| \geq n\}.$$
Then the $P_n$'s are decreasing and have empty intersection. Refine $\cP_n$ to $\beta_n = \{X \setminus P_n, B_n^0, B_n^1, B_n^2\}$ where for $i \in \{0, 1, 2\}$
$$B_n^i = \{x \in P_n \: \ell(x)(n) = i\}.$$
For $n \in \N$ define
$$\gamma_n = \bigvee_{k \leq n} \beta_k.$$
Since each restriction $L(y, \cdot) : \xi \rightarrow \{0, 1, 2\}^{< \omega}$ is essentially injective we have that
\begin{equation} \label{EQN ALPHA}
\xi \subseteq \cF \vee \bigvee_{n \in \N} \salg(\gamma_n).
\end{equation}

Fix $0 < \delta < \min(1/4, \epsilon / 2)$ with
$$-\delta \cdot \log(\delta) - (1 - \delta) \cdot \log(1 - \delta) + \delta \cdot \log(7) < \epsilon.$$
Since
$$\sum_{n \in \N} \mu(P_n) = \int_X |\ell(x)| \ d \mu(x) < \infty$$
we may fix $N \in \N$ so that $\sum_{n = N}^\infty \mu(P_n) < \delta$. Observe that in particular $\mu(P_N) < \delta$ and thus
$$\mu(P_N) + \sum_{n = N}^\infty \mu(P_n) < 2 \delta < 1 / 2.$$

For $n \geq N$ we seek to build $\salg_G(\cP_n \vee \gamma_{n-1})$-expressible functions $\theta_n \in [[E_G^X]]$ with $\dom(\theta_n) = P_n$ and
$$\rng(\theta_n) \subseteq X \setminus \left( P_N \cup \bigcup_{k = N}^{n-1} \theta_k(P_k) \right).$$
We build the $\theta_n$'s by induction on $n \geq N$. To begin we note that $\mu(P_N) < \mu(X \setminus P_N)$ and we apply Lemma \ref{LEM SIMPLEMIX} to obtain $\theta_N$. Now assume that $\theta_N, \ldots, \theta_{n-1}$ have been defined and posses the properties stated above. Then since $\gamma_{n-1}$ refines $\cP_k \vee \gamma_{k-1}$ for every $k < n$, we obtain from Lemma \ref{LEM EXPMOVE}
$$P_N \cup \bigcup_{k = N}^{n-1} \theta_k(P_k) \in \salg_G(\gamma_{n-1}).$$
Also, by our choice of $N$ we have that
\begin{align*}
\mu(P_n) \leq \mu(P_N) < \frac{1}{2} < 1 - 2 \delta & < 1 - \mu(P_N) - \sum_{k = N}^{n-1} \mu(P_k)\\
 & = \mu \left( X \setminus \Big( P_N \cup \textstyle{\bigcup_{k = N}^{n-1}} \, \theta_k(P_k) \Big) \right).
\end{align*}
Therefore we may apply Lemma \ref{LEM SIMPLEMIX} to obtain $\theta_n$. This defines the functions $\theta_n$, $n \geq N$.

Define the partition $\beta = \{X \setminus P, B^0, B^1, B^2\}$ of $X$ by
\begin{align*}
P & = \bigcup_{n \geq N} \theta_n(P_n); \\
B^i & = \bigcup_{n \geq N} \theta_n(B_n^i).
\end{align*}
Note that the above expressions do indeed define a partition of $X$ since the images of the $\theta_n$'s are pairwise disjoint. Also define $\cQ = \{Q, X \setminus Q\}$ where
$$Q = \bigcup_{n \geq N} \theta_n(P_{n+1}).$$
Note that $Q$ is contained in $P$ and so $\beta$ might not refine $\cQ$. Set $\alpha = \gamma_N \vee \beta \vee \cQ$. Then $\alpha$ is finite. Using Lemma \ref{LEM SHAN} and the facts that $X \setminus P \in \beta \vee \cQ$, $\mu(P) < \delta$, and $\sH_{\mu_y}(\gamma_N) \leq \sH_{\mu_y}(\xi)$ for $\nu$-almost-every $y \in Y$ (since $\xi$ $\mu_y$-almost-everywhere refines $\gamma_N$), we obtain
\begin{align*}
\sH(\alpha | \cF) & \leq \sH(\gamma_N | \cF) + \sH(\beta \vee \cQ) \\
 & = \sH(\gamma_N | \cF) + \sH(\{P, X \setminus P\}) + \sH(\beta \vee \cQ | \{P, X \setminus P\}) \\
 & \leq \sH(\gamma_N | \cF) - \mu(P) \cdot \log \mu(P) - \mu(X \setminus P) \log \mu(X \setminus P) + \mu(P) \cdot \log(7) \\
 & < \sH(\gamma_N | \cF) + \epsilon \\
 & = \int_Y \sH_{\mu_y}(\gamma_N) \ d \nu(y) + \epsilon \\
 & \leq \int_Y \sH_{\mu_y}(\xi) \ d \nu(y) + \epsilon \\
 & = \sH(\xi | \cF) + \epsilon.
\end{align*}
Thus it only remains to check that $\salg_G(\alpha) \vee \cF = \salg_G(\xi) \vee \cF$.

First notice that the function $\ell$ and all of the partitions $\gamma_n$ and $\cP_n$ are $\salg_G(\xi) \vee \cF$-measurable and therefore each $\theta_k$ is $\salg_G(\xi) \vee \cF$-expressible. It follows from Lemma \ref{LEM EXPMOVE} that $\beta$, $\cQ$, and $\alpha$ are $\salg_G(\xi) \vee \cF$-measurable. Thus $\salg_G(\alpha) \vee \cF \subseteq \salg_G(\xi) \vee \cF$. Now we consider the reverse inclusion. By induction and by (\ref{EQN ALPHA}) it suffices to assume that $\gamma_k \subseteq \salg_G(\alpha)$ and prove that $\gamma_{k+1} \subseteq \salg_G(\alpha)$ as well. This is immediate when $k \leq N$. So assume that $k \geq N$ and that $\gamma_k \subseteq \salg_G(\alpha)$. Since $\theta_k$ is expressible with respect to $\salg_G(\gamma_k) \subseteq \salg_G(\alpha)$, we have that
$$P_{k+1} = \theta_k^{-1}(Q) \in \salg_G(\alpha)$$
by Lemmas \ref{LEM EXPMOVE} and \ref{LEM EXPGROUP}. Therefore $\cP_{k+1} \subseteq \salg_G(\alpha)$. Now since $\theta_{k+1}$ is expressible with respect to $\salg_G(\cP_{k+1} \vee \gamma_k) \subseteq \salg_G(\alpha)$ we have that for $i \in \{0, 1, 2\}$
$$B_{k+1}^i = \theta_{k+1}^{-1}(B^i) \in \salg_G(\alpha)$$
by Lemmas \ref{LEM EXPMOVE} and \ref{LEM EXPGROUP}. Thus $\beta_{k+1} \subseteq \salg_G(\alpha)$ and we conclude that $\gamma_{k+1} \subseteq \salg_G(\alpha)$. This completes the proof.
\end{proof}

\section{Finite subequivalence relations} \label{SECT EQREL}

The methods of the previous section produce finite generating partitions but do not provide any control over the cardinality or distribution of the partition constructed. Overcoming this difficulty is the main focus of this paper and requires entirely new techniques. We develop these techniques in this section and in Section \ref{SECT ACT}. The goal of this section is to construct finite subequivalence relations which will ultimately be used to replace the traditional role of the Rokhlin lemma and the Shannon--McMillan--Breiman theorem.

For an equivalence relation $E$ on $X$ and $x \in X$, we write $[x]_E$ for the $E$-class of $x$. Recall that a set $T \subseteq X$ is a \emph{transversal} for $E$ if $|T \cap [x]_E| = 1$ for almost-every $x \in X$. We will work with equivalence relations which are generated by an element of the pseudo-group in the following sense.

\begin{defn}
Let $G \acts (X, \mu)$ be a {\pmp} action, let $B \subseteq X$ be a Borel set of positive measure, and let $E$ be an equivalence relation on $B$ with $E \subseteq E_G^X \cap B \times B$. We say that $E$ is \emph{generated by} $\theta \in [[E_G^X]]$ if $\dom(\theta) = \rng(\theta) = B$ and $[x]_E = \{ \theta^i(x) \: i \in \Z\}$ for almost-all $x \in B$. In this case, we write $E = E_\theta$.
\end{defn}

\begin{lem} \label{LEM AVGMIX}
Let $G \acts (X, \mu)$ be an ergodic {\pmp} action, let $B \subseteq X$ have positive measure, let $\alpha$ be a finite partition of $X$, and let $\epsilon > 0$. Then there is an equivalence relation $E$ on $B$ with $E \subseteq E_G^X \cap B \times B$ and $n \in \N$ so that for $\mu$-almost-every $x \in B$, the $E$-class of $x$ has cardinality $n$ and
$$\forall A \in \alpha \qquad \frac{\mu(A \cap B)}{\mu(B)} - \epsilon < \frac{|A \cap [x]_E|}{|[x]_E|} < \frac{\mu(A \cap B)}{\mu(B)} + \epsilon.$$
Moreover, $E$ admits a $\salg_G(\alpha \cup \{B\})$-measurable transversal and is generated by a $\salg_G(\alpha \cup \{B\})$-expressible function $\theta : B \rightarrow B$ in $[[E_G^X]]$ which satisfies $\theta^n = \id_B$.
\end{lem}

\begin{proof}
Let $\pi : (X, \mu) \rightarrow (Y, \nu)$ be the factor map associated to the $G$-invariant sub-$\sigma$-algebra generated by $\alpha \cup \{B\}$. Enumerate $\alpha$ as $\alpha = \{A_1, A_2, \ldots, A_p\}$. Set $B' = \pi(B)$ and $\alpha' = \{A_i' \: 1 \leq i \leq p\}$ where $A_i' = \pi(A_i)$. Note that $\alpha'$ is a partition of $(Y, \nu)$ and that $\nu(A_i' \cap B') = \mu(A_i \cap B)$.

First, let's suppose that $(Y, \nu)$ is non-atomic. Pick $n \in \N$ satisfying $p / n < \epsilon$, and for $1 \leq i \leq p$ set
$$r_i = \lfloor n \nu(A_i' \cap B') / \nu(B') \rfloor.$$
Since $(Y, \nu)$ is non-atomic, we can find a partition $\xi'$ of $B'$ into $n$ pieces each of measure $\frac{1}{n} \cdot \nu(B')$ such that for every $i$, $A_i' \cap B'$ contains at least $r_i$ many classes of $\xi'$. Then at most $p$ many classes of $\xi'$ are not contained in any $A_i' \cap B'$. Set $\xi = \pi^{-1}(\xi')$. Then the classes of $\xi$ lie in $\salg_G(\alpha \cup \{B\})$. Apply Corollary \ref{COR PERMUTE} to get a $\salg_G(\alpha \cup \{B\})$-expressible function $\theta \in [[E_G^X]]$ which cyclically permutes the classes of $\xi$. Set $E = E_\theta$. Then for $\mu$-almost-every $x \in B$, the $E$-class of $x$ has cardinality $n$ and
$$\forall i \qquad - \epsilon \leq - \frac{1}{n} < \frac{|A_i \cap [x]_E|}{|[x]_E|} - \frac{\mu(A_i \cap B)}{\mu(B)} \leq \frac{p}{n} < \epsilon.$$

In the case that $(Y, \nu)$ has an atom, we deduce by ergodicity that, modulo a null set, $Y$ is finite. Say $|Y| = m$ and each point in $Y$ has measure $\frac{1}{m}$. Set $n = |B'|$. Clearly there are integers $k_i \in \N$, with $\sum_{i = 1}^p k_i = n$ and
$$\frac{\mu(A_i \cap B)}{\mu(B)} = \frac{\nu(A_i' \cap B')}{\nu(B')} = \frac{k_i / m}{n / m} = \frac{k_i}{n}.$$
Let $\xi'$ be the partition of $B'$ into points, and pull back $\xi'$ to a partition $\xi$ of $B$. Now apply Corollary \ref{COR PERMUTE} and follow the argument from the non-atomic case.
\end{proof}

\begin{cor} \label{COR AVGFUNCMIX}
Let $G \acts (X, \mu)$ be an ergodic {\pmp} action, let $B \subseteq X$ have positive measure, let $\epsilon > 0$, and let $F = \{ f : B \rightarrow \R \}$ be a finite collection of finite valued Borel functions. Then there is an equivalence relation $E$ on $B$ with $E \subseteq E_G^X \cap B \times B$ and $n \in \N$ so that for $\mu$-almost-every $x \in B$, the $E$-class of $x$ has cardinality $n$ and
$$\forall f \in F \qquad \frac{1}{\mu(B)} \cdot \int_B f \ d \mu - \epsilon < \frac{1}{|[x]_E|} \cdot \sum_{y \in [x]_E} f(y) < \frac{1}{\mu(B)} \cdot \int_B f \ d \mu + \epsilon.$$
Moreover, if each $f \in F$ is $\cF$-measurable then $E$ admits a $\salg_G(\cF \cup \{B\})$-measurable transversal and is generated by a $\salg_G(\cF \cup \{B\})$-expressible function $\theta : B \rightarrow B$ in $[[E_G^X]]$ which satisfies $\theta^n = \id_B$.
\end{cor}

\begin{proof}
Define a partition $\alpha$ of $B$ so that $x, y \in B$ lie in the same piece of $\alpha$ if and only if $f(x) = f(y)$ for all $f \in F$. Then $\alpha$ is a finite partition. Now the desired equivalence relation $E$ is obtained from Lemma \ref{LEM AVGMIX}.
\end{proof}

The conclusions of the previous lemma and corollary are not too surprising since you are allowed to ``see'' the sets which you wish to mix, i.e. you are allowed to use $\salg_G(\alpha \cup \{B\})$. The following proposition however is unexpected. It roughly says that you can achieve the same conclusion even if you are restricted to only seeing a very small sub-$\sigma$-algebra. We will use the proposition below in the same fashion one typically uses the Rokhlin lemma and the Shannon--McMillan--Breiman theorem, although technically the proposition below bears more similarity with the Rokhlin lemma and the ergodic theorem.

Let us say a few words on the Rokhlin lemma to highlight the similarity. For a free {\pmp} action $\Z \acts (X, \mu)$, $n \in \N$, and $\epsilon > 0$, the Rokhlin lemma provides a Borel set $S \subseteq X$ such that the sets $i \cdot S$, $0 \leq i \leq n - 1$, are pairwise disjoint and union to a set having measure at least $1 - \epsilon$. The set $S$ naturally produces a subequivalence relation $E$ defined as follows. For $x \in X$ set $x_S = (-i) \cdot x$ where $(-i) \cdot x \in S$ and $(-j) \cdot x \not\in S$ for all $0 \leq j < i$. We set $x \ E \ y$ if and only if $x_S = y_S$. Clearly every $E$ class has cardinality at least $n$, and a large measure of $E$-classes have cardinality precisely $n$. A key fact which is frequently used in classical results such as Krieger's theorem is that the equivalence relation $E$ is easily described. Specifically, $S$ is small since $\mu(S) \leq 1 / n$, and so $E$ can be defined by using the small sub-$\sigma$-algebra $\salg_\Z(\{S\})$.

\begin{prop} \label{PROP RLEM}
Let $G \acts (X, \mu)$ be an ergodic {\pmp} action with $(X, \mu)$ non-atomic, let $\alpha$ be a finite collection of Borel subsets of $X$, let $\epsilon > 0$, and let $N \in \N$. Then there are $n \geq N$, Borel sets $S_1, S_2 \subseteq X$ with $\mu(S_1) + \mu(S_2) < \epsilon$, and a $\salg_G(\{S_1, S_2\})$-expressible $\theta \in [E_G^X]$ such that $E_\theta$ admits a $\salg_G(\{S_1, S_2\})$-measurable transversal, and for almost-every $x \in X$ we have $|[x]_{E_\theta}| = n$ and
$$\forall A \in \alpha \qquad \mu(A) - \epsilon < \frac{|A \cap [x]_{E_\theta}|}{|[x]_{E_\theta}|} < \mu(A) + \epsilon.$$
\end{prop}

\begin{proof}
Pick $m > \max(4 / \epsilon, \ N)$ with $m \in \N$ and
$$|\alpha| \cdot \log_2(m + 1) < \frac{\epsilon}{4} \cdot m.$$
Let $S_1 \subseteq X$ be any Borel set with $\mu(S_1) = \frac{1}{m} < \frac{\epsilon}{4}$. Apply Corollaries \ref{COR MAKEPART} and \ref{COR PERMUTE} to obtain a $\salg_G(\{S_1\})$-expressible function $h \in [E_G^X]$ such that $\dom(h) = \rng(h) = X$, $h^m = \id_X$, and such that $\{h^i(S_1) \: 0 \leq i < m\}$ is a partition of $X$. The induced Borel equivalence relation $E_h$ is finite, in fact almost-every $E_h$-class has cardinality $m$, and it has $S_1$ as a transversal. We imagine the classes of $E_h$ as extending horizontally to the right, and we visualize $S_1$ as a vertical column.

We consider the distribution of $\alpha \res [s]_{E_h}$ for each $s \in S_1$. For $A \in \alpha$ define $d_A : S_1 \rightarrow \R$ by
$$d_A(s) = \frac{|A \cap [s]_{E_h}|}{|[s]_{E_h}|} = \frac{1}{m} \cdot \Big| A \cap [s]_{E_h} \Big|.$$
Note that for each $A \in \alpha$
$$\int_{S_1} d_A \ d \mu = \frac{1}{m} \cdot \mu(A) = \mu(S_1) \cdot \mu(A).$$
By Corollary \ref{COR AVGFUNCMIX} there is $k \in \N$ and an equivalence relation $E_v \subseteq E_G^X \cap S_1 \times S_1$ on $S_1$ such that for almost every $s \in S_1$, the $E_v$-class of $s$ has cardinality $k$ and
$$\forall A \in \alpha \qquad \mu(A) - \epsilon < \frac{1}{|[s]_{E_v}|} \cdot \sum_{s' \in [s]_{E_v}} d_A(s') < \mu(A) + \epsilon.$$
Moreover, if we let $\cF$ denote the $G$-invariant sub-$\sigma$-algebra generated by the functions $d_A$, $A \in \alpha$, then $E_v$ admits a $\salg_G(\cF \cup \{S_1\})$-measurable transversal $T$ and is generated by a $\salg_G(\cF \cup \{S_1\})$-expressible function $v \in [[E_G^X]]$ which satisfies $\dom(v) = \rng(v) = S_1$ and $v^k = \id_{S_1}$.

Let $E = E_v \vee E_h$ be the equivalence relation generated by $E_v$ and $E_h$. Then $T \subseteq S_1$ is a transversal for $E$, and for every $s \in T$
$$\Big| [s]_E \Big| = \sum_{s' \in [s]_{E_v}} \Big| [s']_{E_h} \Big| = k \cdot m.$$
Setting $n = k \cdot m \geq N$, we have that almost every $E$-class has cardinality $n$. Also, for every $A \in \alpha$ and $s \in T$ we have
$$\frac{|A \cap [s]_E|}{|[s]_E|} = \frac{1}{k \cdot m} \cdot \sum_{s' \in [s]_{E_v}} \Big| A \cap [s']_{E_h} \Big| = \frac{1}{|[s]_{E_v}|} \cdot \sum_{s' \in [s]_{E_v}} d_A(s').$$
It follows that for $\mu$-almost-every $x \in X$
$$\forall A \in \alpha \qquad \mu(A) - \epsilon < \frac{|A \cap [x]_E|}{|[x]_E|} < \mu(A) + \epsilon.$$

Now consider the partition $\xi = \{T_{i,j} \: 0 \leq i < k, \ 0 \leq j < m\}$ of $X$ where
$$T_{i,j} = h^j \circ v^i (T).$$
Note that $T_{i,j} \in \salg_G(\cF \cup \{S_1\})$ by Lemmas \ref{LEM EXPMOVE} and \ref{LEM EXPGROUP}. We will define a function $\theta \in [E_G^X]$ which generates $E$ by defining $\theta$ on each piece of $\xi$. We define
$$\theta \res T_{i,j} = \begin{cases}
h \res T_{i,j} & \text{if } j + 1 < m \\
v \circ h \res T_{i,j} & \text{if } j + 1 = m.
\end{cases}$$
In regard to the second case above, one should observe that $h(T_{i,m-1}) = T_{i,0}$ since $h^m = \id_X$. Since $v$ satisfies $v^k = \id_{S_1}$ and $n = k \cdot m$, we see that $\theta$ satisfies $\theta^n = \id_X$. We also have $E = E_\theta$. Finally, $\theta$ is $\salg_G(\cF \cup \{S_1\})$-expressible since each restriction $\theta \res T_{i, j}$ is $\salg_G(\cF \cup \{S_1\})$-expressible by Lemma \ref{LEM EXPGROUP}.

To complete the proof, we must find a Borel set $S_2 \subseteq X$ with $\mu(S_2) < \frac{3}{4} \cdot \epsilon < \epsilon - \mu(S_1)$ such that $\cF \subseteq \salg_G(\{S_1, S_2\})$. Notice that $|\rng(d_A)| \leq m + 1$ for every $A \in \alpha$ and therefore the product map
$$d_\alpha = \prod_{A \in \alpha} d_A : S_1 \rightarrow \Big\{ 0, \frac{1}{m}, \frac{2}{m}, \ldots, 1 \Big\}^\alpha$$
has an image of cardinality at most $(m + 1)^{|\alpha|}$. Set $\ell = \lceil (\epsilon / 4) \cdot m \rceil$ (i.e. the least integer greater than or equal to $(\epsilon / 4) \cdot m$). Since $(\epsilon / 4) \cdot m  > 1$ we have that $\ell < (\epsilon / 2) \cdot m$. By our choice of $m$ we have
$$(m + 1)^{|\alpha|} < 2^{(\epsilon / 4) \cdot m} \leq 2^\ell.$$
Therefore there is an injection
$$r : \{0, 1/m, \ldots, 1\}^\alpha \rightarrow \{0, 1\}^\ell.$$
Now we will define $S_2$ so that, for every $s \in S_1$, the integers $\{1 \leq i \leq \ell \: h^i(s) \in S_2\}$ will encode the value $r \circ d_\alpha(s)$. Specifically, we define
$$S_2 = \{h^i(s) \: 1 \leq i \leq \ell, \ s \in S_1, \ r(d_\alpha(s))(i) = 1\}.$$
We have that $S_2 \subseteq \bigcup_{1 \leq i \leq \ell} h^i(S_1)$ and therefore
$$\mu(S_2) \leq \ell \cdot \mu(S_1) < \left( \frac{\epsilon}{2} \cdot m \right) \cdot \frac{1}{m} = \frac{\epsilon}{2}$$
as required. Finally, we check that $\cF \subseteq \salg_G(\{S_1, S_2\})$. Fix $p \in \{0, 1/m, \ldots, 1\}^\alpha$. Set
$$I_p^0 = \{1 \leq i \leq \ell \: r(p)(i) = 0\} \quad \text{and} \quad I_p^1 = \{1 \leq i \leq \ell \: r(p)(i) = 1\}.$$
Then for $s \in S_1$ we have
\begin{align*}
d_\alpha(s) = p & \Longleftrightarrow r(d_\alpha (s)) = r(p) \\
 & \Longleftrightarrow (\forall i \in I_p^0) \ \ h^i(s) \not\in S_2 \quad \text{and} \quad (\forall i \in I_p^1) \ \ h^i(s) \in S_2 \\
 & \Longleftrightarrow s \in S_1 \cap \left( \bigcap_{i \in I_p^0} h^{-i}(X \setminus S_2) \right) \cap \left( \bigcap_{i \in I_p^1} h^{-i}(S_2) \right).
\end{align*}
So $d_\alpha^{-1}(p) \in \salg_G(\{S_1, S_2\})$ by Lemmas \ref{LEM EXPMOVE} and \ref{LEM EXPGROUP}. Thus $\cF \subseteq \salg_G(\{S_1, S_2\})$.
\end{proof}

\section{Distributions on finite sets}

In this section we present a few counting lemmas from information theory which we will need. These facts are well known and were used in classical proofs of Krieger's finite generator theorem. At the end of this section we will briefly sketch why replacing the Rokhlin lemma and the Shannon--McMillan--Breiman theorem in the classical proof of Krieger's theorem with Proposition \ref{PROP RLEM} does not (yet) result in a proof of our main theorem. This will illustrate what new techniques are required and will motivate the technical constructions in the next section.

For a finite probability vector $\pv$, $n \in \N$, and $\epsilon \geq 0$, we let $L_{\pv,\epsilon}^n$ be the set of functions $\ell : \{0, \ldots, n-1\} \rightarrow \{0, \ldots, |\pv|-1\}$ which approximate the distribution of $\pv$ in the sense that
$$\forall 0 \leq t < |\pv| \qquad \left| \frac{|\ell^{-1}(t)|}{n} - p_t \right| \leq \epsilon.$$
Similarly, if $(X, \mu)$ is a probability space and $\xi$ is a finite partition of $X$, then we let $L_{\xi, \epsilon}^n$ be the set of functions $\ell : \{0, 1, \ldots, n - 1\} \rightarrow \xi$ such that
$$\forall C \in \xi \qquad \left| \frac{|\ell^{-1}(C)|}{n} - \mu(C) \right| \leq \epsilon.$$
If $\xi$ is a finite partition of $(X, \mu)$ and $\theta : (X, \mu) \rightarrow (X, \mu)$ is a measure-preserving bijection with every $\theta$-orbit having cardinality $n$, then we associate to each $x \in X$ its \emph{$(\xi, \theta)$-name} $\nm_\xi^\theta(x) \in L_{\xi, \infty}^n$ defined by setting $\nm_\xi^\theta(x)(i) = C$ if $\theta^i(x) \in C \in \xi$.

If $\xi$ and $\zeta$ are finite partitions of $(X, \mu)$ and $\xi$ is finer than $\zeta$, then we define the \emph{coarsening map} $\pi_\zeta : \xi \rightarrow \zeta$ to be the unique map satisfying $C \subseteq \pi_\zeta(C)$ for all $C \in \xi$. By applying $\pi_\zeta$ coordinate-wise, we obtain a map $\pi_\zeta : L_{\xi, \infty}^n \rightarrow L_{\zeta, \infty}^n$.

\begin{lem} \label{LEM RELSTIRLING}
Let $(X, \mu)$ be a probability space and let $\xi$ and $\zeta$ be finite partitions of $X$. Suppose that $\xi$ refines $\zeta$ and let $\pi_\zeta : \xi \rightarrow \zeta$ be the coarsening map. Then for every $\kappa > 0$ there is $\epsilon_0 > 0$ so that for all $0 < \epsilon < \epsilon_0$, all sufficiently large $n$, and every $z \in L_{\zeta, \epsilon}^n$
$$\exp \Big( n \cdot \sH(\xi | \zeta) - n \cdot \kappa \Big) \leq \Big| \Big\{ c \in L_{\xi, \epsilon}^n \: \pi_\zeta(c) = z \Big\} \Big| \leq \exp \Big( n \cdot \sH(\xi | \zeta) + n \cdot \kappa \Big).$$
\end{lem}

\begin{proof}
This is a well known fact from information theory which can be quickly deduced from Stirling's formula. See \cite[Lemma 2.13]{CsKo}.
\end{proof}

By taking $\zeta$ to be the trivial partition in the previous lemma, we obtain the following.

\begin{cor} \label{COR STIRLING}
Let $\pv$ be a finite probability vector. Then for every $\kappa > 0$ there is $\epsilon_0 > 0$ so that for all $0 < \epsilon < \epsilon_0$ and all sufficiently large $n$
$$\exp \Big( n \cdot \sH(\pv) - n \cdot \kappa \Big) \leq \Big| L_{\pv, \epsilon}^n \Big| \leq \exp \Big( n \cdot \sH(\pv) + n \cdot \kappa \Big).$$
\end{cor}

For $x \in \R$ we write $\lfloor x \rfloor$ and $\lceil x \rceil$ for the greatest integer less than or equal to $x$ and the least integer greater than or equal to $x$, respectively.

\begin{cor} \label{COR CHOOSE}
Fix $0 < \delta < 1$. Then for every $\kappa > 0$ and for all sufficiently large $n$ we have
$$\binom{n}{\lfloor \delta \cdot n \rfloor} \leq \exp \Big( n \cdot \sH(\delta, 1 - \delta) + n \cdot \kappa\Big).$$
\end{cor}

\begin{proof}
Set $\pv = (1 - \delta, \delta)$. By definition $\binom{n}{\lfloor \delta \cdot n \rfloor}$ is the number of subsets of $\{0, \ldots, n-1\}$ having cardinality $\lfloor \delta \cdot n \rfloor$. Such subsets naturally correspond, via their characteristic functions, to elements of $L_{\pv, \epsilon}^n$ when $n > 1 / \epsilon$. Thus when $n > 1 / \epsilon$ we have $\binom{n}{\lfloor \delta \cdot n \rfloor} \leq |L_{\pv, \epsilon}^n|$. Now apply Corollary \ref{COR STIRLING}.
\end{proof}

The normalized Hamming metric $\dHam$ on the set $L_{\pv, \infty}^n$ is defined by
$$\dHam(\ell, \ell') = \frac{1}{n} \cdot | \{i \: \ell(i) \neq \ell'(i)\} |.$$

\begin{cor} \label{COR SEP}
Let $\pv$ be a finite probability vector. Then for every $\kappa > 0$ there are $\delta, \epsilon_0 > 0$ so that for all $0 < \epsilon < \epsilon_0$ and all sufficiently large $n$ there exists $K \subseteq L_{\pv, \epsilon}^n$ satisfying $\dHam(k, k') > 2 \delta$ for all $k \neq k' \in K$ and $|K| \geq \exp(n \cdot \sH(\pv) - n \cdot \kappa)$.
\end{cor}

\begin{proof}
Fix $\delta, \epsilon_0 > 0$ so that
$$\sH(2 \delta, 1 - 2 \delta) + 2 \delta \cdot \log |\pv| < \kappa / 3$$
and $|L_{\pv, \epsilon}^n| \geq \exp(n \cdot \sH(\pv) - n \cdot \kappa / 3)$ for all $0 < \epsilon < \epsilon_0$ and for sufficiently large $n$. For $V \subseteq L_{\pv, \infty}^n$ let
$$B(V; \rho) = \{\ell \in L_{\pv, \infty}^n \: \exists v \in V \ \ \dHam(\ell, v) \leq \rho\}$$
be the ball about $V$ of radius $\rho$. Basic combinatorics implies that for $0 < \rho < 1$
$$\Big| B(V; \rho) \Big| \leq |V| \cdot \binom{n}{\lfloor \rho \cdot n \rfloor} \cdot |\pv|^{\rho \cdot n}.$$
Fix $0 < \epsilon < \epsilon_0$. For each $n$ let $K_n \subseteq L_{\pv, \epsilon}^n$ be maximal with the property that $\dbar(k, k') > 2 \delta$ for all $k \neq k' \in K_n$. Then by maximality of $K_n$ we have $L_{\pv, \epsilon}^n \subseteq B(K_n; 2 \delta)$ and thus
$$| L_{\pv, \epsilon}^n | \leq | B(K_n; 2 \delta) | \leq |K_n| \cdot \binom{n}{\lfloor 2 \delta \cdot n \rfloor} \cdot |\pv|^{2 \delta \cdot n}.$$
Solving for $|K_n|$ and letting $n$ be sufficiently large gives
\begin{align*}
|K_n| & \geq | L_{\pv, \epsilon}^n | \cdot \binom{n}{\lfloor 2 \delta \cdot n \rfloor}^{-1} \cdot |\pv|^{- 2 \delta \cdot n}\\
 & \geq \exp(n \cdot \sH(\pv) - n \cdot \kappa / 3 - n \cdot \sH(2 \delta, 1 - 2 \delta) - n \cdot \kappa / 3 - n \cdot 2 \delta \cdot \log |\pv|)\\
 & \geq \exp(n \cdot \sH(\pv) - n \kappa).\qedhere
\end{align*}
\end{proof}

We present one more technical lemma we will need.

\begin{lem} \label{LEM J}
Let $\pv$ be a finite probability vector, let $\epsilon, \delta > 0$, and let $n \in \N$. If $\ell \in L_{\pv, \epsilon}^n$ then there is $J \subseteq \{0, \ldots, n-1\}$ such that $|J| \leq (\epsilon n + 1) \cdot |\pv| + \delta n$ and
$$\forall 0 \leq t < |\pv| \qquad \frac{1}{n} \cdot \Big| \{ i : \ell(i) = t\} \setminus J \Big| \leq (1 - \delta) p_t.$$
\end{lem}

\begin{proof}
For each $t$ we have $|\{i : \ell(i) = t\}| \leq n p_t + n \epsilon$. So we may choose $J$ so that for every $t$
$$|J \cap \{i : \ell(i) = t\}| = \max(0, \ \lceil |\{i : \ell(i) = t\}| - (1 - \delta) n p_t \rceil) \leq \lceil \epsilon n + \delta n p_t\rceil.$$
Then $J$ will have the desired property and
\begin{equation*}
|J| \leq \sum_{t = 0}^{|\pv| - 1} \lceil \epsilon n + \delta n p_t \rceil \leq \epsilon n \cdot |\pv| + \delta n + |\pv| = (\epsilon n + 1) |\pv| + \delta n.\qedhere
\end{equation*}
\end{proof}

Before closing this section we briefly clarify to the reader what new methods are required in order to prove the main theorem. Let us consider the simplest setting where partitions and probability vectors are finite, $\cF = \{X, \varnothing\}$ is trivial, and $r = 1$. The argument we present below is simply a sketch intended to give some intuition and motivation.

Let $G \acts (X, \mu)$ be a {\pmp} ergodic action with $\mu$ non-atomic, let $\xi$ be a finite generating partition, and let $\pv$ be a finite probability vector with $\sH(\xi) < \sH(\pv)$. We would like to construct a generating partition $\alpha = \{A_i : 0 \leq i < |\pv|\}$ with $\mu(A_i) = p_i$ for all $i$. Pick $n_0 \in \N$ and $\epsilon > 0$. By Proposition \ref{PROP RLEM} there are $n \geq n_0$, Borel sets $S_1, S_2 \subseteq X$ with $\mu(S_1) + \mu(S_2) < \epsilon$, and a $\salg_G(\{S_1, S_2\})$-expressible $\theta \in [E_G^X]$ such that $E_\theta$ admits a $\salg_G(\{S_1, S_2\})$-measurable transversal $Y$ and such that for $\mu$-almost-every $x \in X$, the $E_\theta$ class of $x$ has cardinality $n$, and
$$\forall C \in \xi  \qquad \mu(C) - \epsilon < \frac{|C \cap [x]_{E_\theta}|}{|[x]_{E_\theta}|} < \mu(C) + \epsilon.$$
So we have $\nm_\xi^\theta(y) \in L_{\xi, \epsilon}^n$ for almost-every $y \in Y$.

Since $|L_{\xi, \epsilon}^n| \approx \exp(n \cdot \sH(\xi)) < \exp(n \cdot \sH(\pv)) \approx |L_{\pv,\epsilon}^n|$, by beginning with a sufficiently small $\epsilon$ and sufficiently large $n_0$, we can conclude from Corollary \ref{COR STIRLING} that
$$|L_{\xi,\epsilon}^n| < |L_{\pv,\epsilon}^n|.$$
So there is an injection $f : L_{\xi,\epsilon}^n \rightarrow L_{\pv, \epsilon}^n$. For $y \in Y$ set $c_y = \nm_\xi^\theta(y)$ and $a_y = f(c_y)$. Since $\{\theta^i(Y) : 0 \leq i < n\}$ is a partition of $X$, there is a unique partition $\alpha = \{A_i : 0 \leq i < |\pv|\}$ of $X$ satisfying $\nm_\alpha^\theta(y) = a_y$ for all $y \in Y$. Specifically, $x \in A_t$ if and only if $a_y(i) = t$ where $y \in Y$ and $0 \leq i < n$ satisfy $\theta^i(y) = x$. The condition $\nm_\alpha^\theta(y) = a_y \in L_{\pv, \epsilon}^n$ implies that $|\mu(A_i) - p_i| \leq \epsilon$ for all $i$. Since $f$ is an injection, $c_y = f^{-1}(a_y)$ is determined from $a_y$. It can be deduced from this fact that
\begin{equation} \label{EQN FAIL}
\xi \subseteq \salg_{\langle \theta \rangle}(\alpha \vee \{Y, X \setminus Y\}).
\end{equation}

Now it becomes clear what is missing in order to prove our main theorem. We do want $\mu(A_i) = p_i$ instead of $|\mu(A_i) - p_i| \leq \epsilon$, but this is only a minor problem. The major problem is that instead of (\ref{EQN FAIL}) we want $\xi \subseteq \salg_G(\alpha)$, so that $\alpha$ is a generating partition. For this, it would suffice to have both (\ref{EQN FAIL}) hold and $S_1, S_2 \in \salg_G(\alpha)$. In this case, $Y$ would be measurable and $\theta$ would be expressible with respect to $\salg_G(\{S_1, S_2\}) \subseteq \salg_G(\alpha)$ and therefore
$$\xi \subseteq \salg_{\langle \theta \rangle}(\alpha \vee \{Y, X \setminus Y\}) \subseteq \salg_G(\alpha).$$
So we want to have both (\ref{EQN FAIL}) hold and $S_1, S_2 \in \salg_G(\alpha)$ simultaneously. This first requirement on $\alpha$ uses $\theta$-translates and the second uses $G$-translates. The difficulty is that we must build an $\alpha$ which simultaneously encodes messages under the $\theta$-action and encodes messages under the $G$-action, and these two actions are almost completely unrelated. We solve this problem in the next section.

\section{Coding small sets} \label{SECT ACT}

The goal of this section is to construct, for any two sets $S_1, S_2$ having small measure, a pre-partition $\beta = \{B_0, B_1\}$ with the property that $\mu(\cup \beta)$ is small and $S_1, S_2 \in \ralg_G(\beta)$. This task is vital to the proof of the main theorem, as it will connect the $G$-action and the transformation $\theta$ used in Proposition \ref{PROP RLEM}.

We first focus our attention on building a pre-partition $\beta$ and a set $R$, $0 < \mu(R) < 1$, with $R \in \ralg_G(\beta)$. The pre-partition $\beta$ will consist of two (disjoint) sets $B_0, B_1$. On an intuitive level, it is likely helpful to imagine points in $B_0$ as ``labeled with $0$,'' points in $B_1$ as ``labeled with $1$'', and points in $X \setminus (B_0 \cup B_1)$ as ``unlabeled.'' The condition $R \in \ralg_G(\beta)$ then roughly means that no matter how the unlabeled points are later labeled, for every $x \in X$ it is possible to determine from the labeling of its orbit whether $x \in R$. This condition is similar to the notions of locally recognizable functions, membership tests, and recognizable sets appearing in \cite{GJS09,GJS12,ST14}. It is the similarity with recognizable sets which led us to use the letter $R$.

\begin{figure}[ht]
\begin{center}
\setlength{\unitlength}{4mm}
\begin{picture}(15,11)(0,0)

\put(7.5,0.75){\makebox(0,0)[b]{$F$}}
\qbezier(0,4.5)(0,11)(10.5,11)
\qbezier(0,4.5)(0,0)(7.5,0)
\qbezier(7.5,0)(16,0)(16,7.5)
\qbezier(10.5,11)(16,11)(16,7.5)

\put(0,0){
\put(4.5,3.25){\makebox(0,0)[b]{$W$}}
\qbezier(3,4.5)(3,3)(4.5,3)
\qbezier(3,4.5)(3,6)(4.5,6)
\qbezier(4.5,3)(6,3)(6,4.5)
\qbezier(4.5,6)(6,6)(6,4.5)
\put(4.5,4.35){\makebox(0,0)[b]{$\bullet$}\makebox(1.25,1.75)[b]{$1_G$}}

\put(4.5,7.5){\makebox(0,0)[b]{$\bullet$}\makebox(1.25,1.75)[b]{$c$}}
}

\put(-0.5,1){
\put(9,5.75){\makebox(0,0)[b]{$\bullet$}\makebox(1.25,1.75)[b]{$q_1$}}
\put(9,3.25){\makebox(0,0)[b]{$\bullet$}\makebox(1.25,1.75)[b]{$q_2$}}
\put(11.25,2){\makebox(0,0)[b]{$\bullet$}\makebox(1.25,1.75)[b]{$q_3$}}
\put(13.5,3.25){\makebox(0,0)[b]{$\bullet$}\makebox(1.25,1.75)[b]{$q_4$}}
\put(13.5,5.75){\makebox(0,0)[b]{$\bullet$}\makebox(1.25,1.75)[b]{$q_5$}}
\put(11.25,7){\makebox(0,0)[b]{$\bullet$}\makebox(1.25,1.75)[b]{$q_6$}}
}

\end{picture}
\caption{\label{fig:WF}The structure in $G$ which will be needed in constructing $R$ and $\beta = \{B_0, B_1\}$.}
\end{center}
\end{figure}
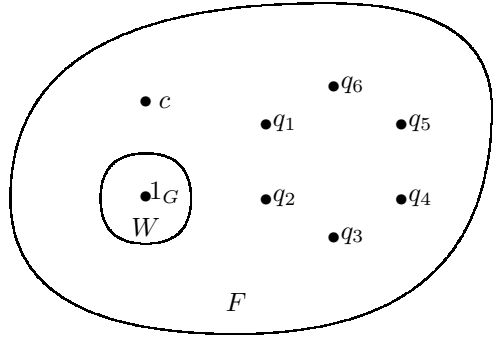

A naive but suggestive idea for building $R, B_0, B_1$ is to fix a finite window $W \subseteq G$ and a finite set $F$ containing $W^2$, and label all points in $W \cdot R$ with $1$ (i.e. set $B_1 = W \cdot R$), label all points in $(F \setminus W) \cdot R$ with $0$, and label more points $0$ as needed so that for every $x \not\in R$ there is a point in $W \cdot x$ labeled $0$. It seems plausible then that $x \in R$ if and only if $W \cdot x$ is labeled identically $1$. If so, $R \in \ralg_G(\beta)$ as desired. This naive approach is the right idea but does not quite work. For example, this may fail if $W$ has too much symmetry, such as if $W$ is a finite subgroup, for one might not be able to distinguish $R$ from $W \cdot R$. This problem can be easily fixed by wisely choosing a ``checkpoint'' $c \in F \setminus W$ and labeling $c \cdot R$ with $1$ in order to break any potential symmetries. In the case of free actions nothing more is required, but for non-free actions additional problems emerge which are a bit tedious to handle. There are a few interacting problems in the case of a non-free action, but in brief the primary problem is that we may have $W \cdot x = \{x\}$ for some $x = c \cdot r$ with $r \in R$. We will overcome this problem by introducing two new points $q_1, q_2 \in F \setminus W$. In fact, we will construct a set $Q = \{q_1, \ldots, q_6\} \subseteq F \setminus W$ of ``query points'' whose labels will hold important information. However $q_3, \ldots, q_6$ will not be needed until the second half of this section. See Figure \ref{fig:WF} for an illustration.

The following lemma is well known.

\begin{lem} \label{LEM MARKER}
Let $G \acts (X, \mu)$ be a {\pmp} action. If $Y \subseteq X$ is Borel and $F \subseteq G$ is finite, then there exists a Borel set $D \subseteq Y$ such that $Y \subseteq F^{-1} F \cdot D$ and $F \cdot d \cap F \cdot d' = \varnothing$ for all $d \neq d' \in D$. In particular, if $\mu(Y) > 0$ then $\mu(D) > 0$.
\end{lem}

\begin{proof}
Define the Borel graph $\Gamma \subseteq Y \times Y$ by $(y, y') \in \Gamma$ if and only if $y \neq y'$ and $F \cdot y \cap F \cdot y' \neq \varnothing$. Since every vertex has finite degree in $\Gamma$, a result of Kechris--Solecki--Todorcevic \cite[Prop 4.2 and Prop. 4.5]{KST99} states that there is a maximal (with respect to containment) Borel set $D$ which is $\Gamma$-independent (i.e. no two elements of $D$ are adjacent). Since $D$ is $\Gamma$-independent we have $F \cdot d \cap F \cdot d' = \varnothing$ for all $d \neq d'$, and since $D$ is maximal we have $Y \subseteq F^{-1} F \cdot D$.
\end{proof}

In the outline above we mentioned labeling points $0$ as needed so that for most $x \in X$ there is a point in $W \cdot x$ labeled $0$. The lemma below will help us determine where to place these $0$'s.

\begin{lem} \label{LEM SMARKER}
Let $G \acts (X, \mu)$ be a {\pmp} ergodic action with $(X, \mu)$ non-atomic, and let $\delta > 0$. Then there exists a finite symmetric set $W \subseteq G$ with $1_G \in W$ and a Borel set $D \subseteq X$ such that $0 < \mu(D) < \delta$ and $W \cdot x \cap D \neq \varnothing$ for all $x \in X$.
\end{lem}

\begin{proof}
Since $(X, \mu)$ is non-atomic, we can fix a Borel set $D_1 \subseteq X$ with $0 < \mu(D_1) < \delta / 2$. As $G \acts (X, \mu)$ is ergodic, we have $\mu(G \cdot D_1) = 1$. So there must be a sufficiently large finite symmetric set $W \subseteq G$ with $1_G \in W$ satisfying $\mu(W \cdot D_1) > 1 - \delta/2$. Now set $D = D_1 \cup (X \setminus W \cdot D_1)$. Then $\mu(D) < \delta$ and $W \cdot D = X$.
\end{proof}

We will need the following fact from group theory.

\begin{lem}[B.H. Neumann, \cite{N}] \label{LEM BHN}
Let $G$ be a group, and let $H_i$, $1 \leq i \leq n$, be subgroups of $G$. Suppose there are group elements $g_i \in G$ so that
$$G = \bigcup_{i = 1}^n g_i \cdot H_i.$$
Then there is $i$ such that $|G : H_i| < \infty$.
\end{lem}

The corollary below will allow us to construct the checkpoint $c$.

\begin{cor} \label{COR APPBHN}
Let $G \acts (X, \mu)$ be a {\pmp} ergodic action with $(X, \mu)$ non-atomic. Let $R \subseteq X$ have positive measure and let $W, T \subseteq G$ be finite. Then there are a Borel set $R' \subseteq R$ with $\mu(R') > 0$ and $c \in G$ such that $c W \cdot R' \cap T \cdot R' = \varnothing$.
\end{cor}

\begin{proof}
Our assumptions imply that almost-every orbit is infinite. So for $\mu$-almost-every $r \in R$ the stability group $\Stab(r) = \{g \in G \: g \cdot r = r\}$ has infinite index in $G$ and thus by Lemma \ref{LEM BHN}
$$T \cdot \Stab(r) \cdot W^{-1} = \bigcup_{t \in T} \bigcup_{w \in W} t w^{-1} \cdot (w \Stab(r) w^{-1}) \neq G.$$
As $G$ is countable, there is $c \in G$ and a non-null Borel set $R_0 \subseteq R$ with
$$c \not\in T \cdot \Stab(r) \cdot W^{-1}$$
for all $r \in R_0$. It follows that $c W \cdot r \cap T \cdot r = \varnothing$ for all $r \in R_0$. Now apply Lemma \ref{LEM MARKER} to get positive measure Borel set $R' \subseteq R_0$ with $(c W \cup T) \cdot r \cap (c W \cup T) \cdot r' = \varnothing$ for all $r \neq r' \in R'$.
\end{proof}

The next lemma will later be applied six times in order to build $Q = \{q_1, \ldots, q_6\}$.

\begin{lem} \label{LEM AVOID}
Let $G \acts (X, \mu)$ be a {\pmp} ergodic action with $(X, \mu)$ non-atomic. Let $R, Y \subseteq X$ be positive measure Borel sets and let $T \subseteq G$ be finite. Then there are $q \in G$ and a Borel set $R' \subseteq R$ of positive measure such that $q \cdot R' \subseteq Y$ and $q \cdot R' \cap T \cdot R' = \varnothing$.
\end{lem}

\begin{proof}
Let $R_0 \subseteq R$ be a Borel set with $\mu(R_0) > 0$ and $\mu(Y \setminus T \cdot R_0) > 0$. By ergodicity, there is $q \in G$ such that $R_0 \cap q^{-1} \cdot (Y \setminus T \cdot R_0)$ has positive measure. Set
\begin{equation*}
R' = R_0 \cap q^{-1} \cdot (Y \setminus T \cdot R_0). \qedhere
\end{equation*}
\end{proof}

Now we are ready to build $W, F \subseteq G$ and $c, q_1, \ldots, q_6 \in G$ as pictured in Figure \ref{fig:WF}. We mention that when we will later apply this lemma, the set $Y$ will be defined as $Y = \{y \in X : |W \cdot y| > 1\}$. We also mention that the equalities in clauses (v) and (vi) below are mostly due to the fact that $1_G \in W$, while in clause (vii) it would be preferable that the intersection be empty but due to non-trivial stabilizers the stated containment is the best one can hope for. Following this lemma we will immediately build a pre-partition $\beta$ with $R \in \ralg_G(\beta)$.

\begin{lem} \label{LEM TECH}
Let $G \acts (X, \mu)$ be a {\pmp} ergodic action with $(X, \mu)$ non-atomic. Let $Y \subseteq X$ be a Borel set of positive measure, let $W \subseteq G$ be finite and symmetric with $1_G \in W$, and let $m \in \N$. Then there exist $n \in \N$, $F \cup Q \cup \{c\} \subseteq G$, and a Borel set $R \subseteq X$ with $Q = \{q_1, \ldots, q_6\}$, $\mu(R) = \frac{1}{n}$, $n > m \cdot |F|$, and satisfying the following:
\begin{enumerate}
\item[\rm (i)] $Q \cdot R \subseteq Y$;
\item[\rm (ii)] $|(\{c\} \cup Q) \cdot r \setminus W \cdot r| = 7$ for all $r \in R$;
\item[\rm (iii)] $( W \cup \{c\} \cup Q )^2 \subseteq F$;
\item[\rm (iv)] $F \cdot r \cap F \cdot r' = \varnothing$ for all $r \neq r' \in R$;
\item[\rm (v)] $W q \cdot R \cap (W \cup \{c\} \cup Q) \cdot R = q \cdot R$ for every $q \in Q$;
\item[\rm (vi)] $c W \cdot R \cap (W \cup \{c\} \cup Q) \cdot R = c \cdot R$;
\item[\rm (vii)] $Q c \cdot R \cap (W \cup \{c\} \cup Q) \cdot R \subseteq c \cdot R$;
\item[\rm (viii)] for all $r \in R$, either $q_1 c \cdot r \neq c \cdot r$ or $q_2 c \cdot r = c \cdot r$.
\end{enumerate}
\end{lem}

\begin{proof}
Set $R_0 = X$. By induction on $1 \leq i \leq 6$ we choose $q_i \in G$ and a Borel set $R_i \subseteq R_{i-1}$ such that $\mu(R_i) > 0$, $q_i \cdot R_i \subseteq Y$, and
$$q_i \cdot R_i \cap W (W \cup \{q_j \: j < i\}) \cdot R_i = \varnothing.$$
Both the base case and the inductive steps are taken care of by Lemma \ref{LEM AVOID}. Set $Q = \{q_1, q_2, \ldots, q_6\}$. Then $q_i \cdot R_6 \subseteq q_i \cdot R_i \subseteq Y$ and $|Q \cdot r \setminus W \cdot r| = 6$ for all $r \in R_6$. Now apply Corollary \ref{COR APPBHN} to obtain $c \in G$ and a Borel set $R_c \subseteq R_6$ with $\mu(R_c) > 0$ and
$$c W \cdot R_c \cap (\{1_G\} \cup Q^{-1})(W \cup Q \cup W Q) \cdot R_c = \varnothing.$$
Set $F = (W \cup \{c\} \cup Q)^2$ so that (iii) is satisfied.

If there is $q \in Q$ with $q c \cdot r = c \cdot r$ for all $r \in R_c$, then set $R' = R_c$ and re-index the elements of $Q$ so that $q_2 = q$. Otherwise, we may re-index $Q$ and find a Borel set $R' \subseteq R_c$ of positive measure with $q_1 c \cdot r \neq c \cdot r$ for all $r \in R'$. Now apply Lemma \ref{LEM MARKER} to obtain a positive measure Borel set $R \subseteq R'$ with $F \cdot r \cap F \cdot r' = \varnothing$ for all $r \neq r' \in R$. By shrinking $R$ if necessary, we may suppose that $\mu(R) = \frac{1}{n}$ for some $n > m \cdot |F|$. Then (iv) is immediately satisfied, (viii) is satisfied since $R \subseteq R'$, and (i) is satisfied since $R \subseteq R_6$. Clause (ii) also holds since $c \cdot r \in c W \cdot r$ is disjoint from $(W \cup Q) \cdot r$ for every $r \in R$.

Recall that $W = W^{-1}$ and $1_G \in W$. Fix $1 \leq i \leq 6$. By the definition of $q_i$ we have $W q_i \cdot R \cap W \cdot R = \varnothing$, and if $j \neq i$ then $W q_i \cdot R \cap q_j \cdot R = \varnothing$. Also, the definition of $c$ implies that $W q_i \cdot R \cap c \cdot R = \varnothing$. Therefore
$$W q_i \cdot R \cap (W \cup \{c\} \cup Q) \cdot R \subseteq q_i \cdot R.$$
This establishes (v) since $1_G \in W$. By definition of $c$ we have $Q c \cdot R \cap (W \cup Q) \cdot R = \varnothing$. So (vii) follows. Similarly, $c W \cdot R \cap (W \cup Q) \cdot R = \varnothing$ which gives one inclusion in (vi), and the reverse inclusion follows from $1_G \in W$.
\end{proof}

Now we construct a pre-partition $\beta$ with $R \in \ralg_G(\beta)$.

\begin{lem} \label{LEM REC}
Let $G \acts (X, \mu)$ be a {\pmp} ergodic action with $(X, \mu)$ non-atomic. Let $W \subseteq G$ be finite and symmetric with $1_G \in W$, and let $D \subseteq X$ be a Borel set with $W \cdot x \cap D \neq \varnothing$ for all $x \in X$. Assume that the set $Y = \{x \in X \: |W \cdot x| \geq 2\}$ has positive measure, and let $F \cup Q \cup \{c\} \subseteq G$ and $R \subseteq X$ be as in Lemma \ref{LEM TECH}. If $\beta = \{B_0, B_1\}$ is a pre-partition satisfying
\begin{align}
B_1 & \supseteq (W \cup \{c, q_1\}) \cdot R, \quad \text{and}\label{eqn:B1}\\
B_0 & \supseteq (D \setminus F \cdot R) \bigcup \Big( F \cdot R \setminus (W \cup \{c\} \cup Q) \cdot R \Big) \bigcup q_2 \cdot R\label{eqn:B0}
\end{align}
then $R \in \ralg_G(\beta)$.
\end{lem}

\begin{proof}
We will use the roman numerals (i) through (viii) to refer to clauses of Lemma \ref{LEM TECH}. If $r \in R$ then it is immediate from the definitions that $(W \cup \{c, q_1\}) \cdot r \subseteq B_1$ and $q_2 \cdot r \in B_0$. So it suffices to show that if $x \not\in R$ then either $(W \cup \{c, q_1\}) \cdot x \cap B_0 \neq \varnothing$ or $q_2 \cdot x \in B_1$.

Fix $x \not\in R$. Let $w \in W$ be such that $w \cdot x \in D$. If $w \cdot x \in B_0$ then we are done. So suppose that $w \cdot x \not\in B_0$. Since $w \cdot x \in D \setminus B_0$, by (\ref{eqn:B0}) we must have $w \cdot x \in F \cdot R$. Now $w \cdot x \in F \cdot R \setminus B_0$ so again by (\ref{eqn:B0}) we obtain $w \cdot x \in (W \cup \{c\} \cup Q) \cdot R$. If $x \in B_0$ then we are done since $1_G \in W$. So suppose that $x \not\in B_0$. Since $W$ is symmetric, $x \in W \cdot w \cdot x$ and hence $x \in F \cdot R$ by (iii). Again, $x \in F \cdot R \setminus B_0$ so (\ref{eqn:B0}) implies $x \in (W \cup \{c\} \cup Q) \cdot R$.

The previous paragraph shows that if $W \cdot x$ is disjoint with $B_0$ then we must have $x \in (W \cup \{c\} \cup Q) \cdot R$. We will divide the remainder of the argument into three cases: $x \in W \cdot R \setminus R$, $x \in c \cdot R$, and $x \in Q \cdot R$. Along the way we will illuminate why $c, q_1, q_2$ are needed in this construction.

Suppose that $x \in W \cdot R \setminus R$. Here is the problem in which $W$ together with stabilizers may possess too much symmetry: even though $x \not\in R$, it may be that $W \cdot x \subseteq W \cdot R \subseteq B_1$. This problem is resolved by using the checkpoint $c$. Since $x \not\in R$, it follows from (vi) that $c \cdot x \not\in (W \cup \{c\} \cup Q) \cdot R$. By (iii) we have $c \cdot x \in F \cdot R$, so by (\ref{eqn:B0}) we find that $c \cdot x \in B_0$. Thus in the case $x \in W \cdot R \setminus R$ we are done.

Now suppose that $x \in c \cdot R$. The problem in this case is that we may have $W \cdot x = \{x\}$, and based on the conditions imposed on $c$ this situation is unavoidable to the best knowledge of the author. We handle this problem by using the points $q_1$ and $q_2$. Fix $r \in R$ with $x = c \cdot r$. By (viii) we have that either $q_1 c \cdot r \neq c \cdot r$ or $q_2 c \cdot r = c \cdot r$. In the latter case, (\ref{eqn:B1}) gives
$$q_2 \cdot x = q_2 c \cdot r = c \cdot r \in B_1$$
and we are done. So assume that $q_1 c \cdot r \neq c \cdot r$. Then $q_1 c \cdot r \not\in c \cdot R$ by (iii) and (iv) and so by (vii)
$$q_1 c \cdot r \not\in (W \cup \{c\} \cup Q) \cdot R.$$
As $q_1 c \in F$ by (iii), we find that
$$q_1 \cdot x = q_1 c \cdot r \in F \cdot R \setminus (W \cup \{c\} \cup Q) \cdot R \subseteq B_0$$
which finishes this case.

Finally, suppose that $x \in Q \cdot R$. Thankfully $Q$ does not create any new problems and the argument can terminate here. Fix $r \in R$ and $q \in Q$ with $x = q \cdot r$. By (i) $q \cdot r \in Y$ and hence there is $w \in W$ with $w q \cdot r \neq q \cdot r$. It follows $w q \cdot r \not\in q \cdot R$ by (iii) and (iv) and so by (v)
$$w q \cdot r \not\in (W \cup \{c\} \cup Q) \cdot R.$$
Therefore
\begin{equation*}
w \cdot x = w q \cdot r \in F \cdot R \setminus (W \cup \{c\} \cup Q) \cdot R \subseteq B_0.\qedhere
\end{equation*}
\end{proof}

We are now ready for the main result of this section.

\begin{prop} \label{PROP CODE}
Let $G \acts (X, \mu)$ be a {\pmp} ergodic action with $(X, \mu)$ non-atomic and let $0 < \delta < 1$. Then there are $\epsilon > 0$ and a Borel set $M \subseteq X$ with $\mu(M) = \delta$ with the following property: for any $S_1, S_2 \subseteq X$ satisfying $\mu(S_1) + \mu(S_2) < \epsilon$ there is a two-piece partition $\beta = \{B_0, B_1\}$ of $M$ with $S_1, S_2 \in \ralg_G(\beta)$.
\end{prop}

\begin{proof}
By Lemma \ref{LEM SMARKER}, there is a finite symmetric set $W \subseteq G$ with $1_G \in W$ and a Borel set $D \subseteq X$ with $\mu(D) < \delta / 2$ such that $W \cdot x \cap D \neq \varnothing$ for all $x \in X$. Note that if $|W \cdot x| = 1$ then $x \in D$. Thus the set $Y = \{x \in X \: |W \cdot x| \geq 2\}$ has positive measure. Apply Lemma \ref{LEM TECH} to obtain $F \cup \{c\} \cup Q \subseteq G$ with $Q = \{q_1, \ldots, q_6\}$ and $R \subseteq X$ with $\mu(R) = \frac{1}{n}$, where $n > 2 |F| / \delta$. Fix $k \in \N$ with
$$\log_2(2 n k) < k - 1$$
and let $Z_1$ and $Z_2$ be disjoint Borel subsets of $R$ with $\mu(Z_1) = \mu(Z_2) = \frac{1}{2 n k}$. Set $Z = Z_1 \cup Z_2$ and note that $\mu(Z) = \frac{1}{n k} = \frac{1}{k} \cdot \mu(R)$. Fix $\epsilon > 0$ with $\epsilon < \frac{1}{6 n k}$. Let $M \subseteq X$ be any Borel set with $D \cup F \cdot R \subseteq M$ and $\mu(M) = \delta$.

Apply Corollaries \ref{COR MAKEPART} and \ref{COR PERMUTE} to obtain a $\salg_G(\{Z, R\})$-expressible function $\rho \in [[E_G^X]]$ such that $\dom(\rho) = \rng(\rho) = R$, $\rho^k = \id_R$, and such that $\{\rho^i(Z) \: 0 \leq i < k\}$ is a partition of $R$. For each $j = 1, 2$, again apply these corollaries to obtain a $\salg_G(\{Z_j\})$-expressible function $\psi_j \in [E_G^X]$ such that $\dom(\psi_j) = \rng(\psi_j) = X$, $\psi_j^{2 n k} = \id_X$, and such that $\{\psi_j^i(Z_j) \: 0 \leq i < 2 n k\}$ is a partition of $X$. We mention that there are no assumed relationships between $\psi_1$, $\psi_2$, and $\rho$.

Let $S_1, S_2 \subseteq X$ be Borel sets with $\mu(S_1) + \mu(S_2) < \epsilon$. Our intention will be to encode how the sets $\psi_1^i(Z_1)$ meet $S_1$ and similarly how the sets $\psi_2^i(Z_2)$ meet $S_2$. For $1 \leq m \leq 2 n k$ and $j = 1, 2$, let $Z_j^m$ be the set of $z \in Z_j$ such that
$$| \{ 0 \leq i < 2 n k \: \psi_j^i(z) \in S_j\} | \geq m.$$
Then $Z_j^1 \supseteq Z_j^2 \supseteq \cdots \supseteq Z_j^{2 n k}$ and
$$\sum_{m = 1}^{2 n k} \mu(Z_1^m \cup Z_2^m) = \mu(S_1) + \mu(S_2) < \epsilon.$$
Setting $Z_j^* = Z_j \setminus Z_j^1$, we have
$$\mu(Z_j^*) = \mu(Z_j) - \mu(Z_j^1) > \frac{1}{2 n k} - \epsilon > \frac{1}{3 n k} > 2 \epsilon.$$
In particular,
\begin{equation} \label{EQN KAPPA}
\mu(Z_1^* \cup Z_2^*) - \sum_{m = 1}^{2 n k} \mu(Z_1^m \cup Z_2^m) > 4 \epsilon - \epsilon = 3 \epsilon.
\end{equation}
Set $Z^m = Z_1^m \cup Z_2^m$ and $Z^* = Z_1^* \cup Z_2^*$.

For each $1 \leq m \leq 2 n k$ we wish to build a function $\theta_m \in [[E_G^X]]$ which is expressible with respect to $\salg_G(\{ Z^*, Z^1, \ldots, Z^m \})$ and satisfies $\dom(\theta_m) = Z^m$ and
$$\rng(\theta_m) \subseteq Z^* \setminus \bigcup_{k = 1}^{m-1} \theta_k(Z^k).$$
We construct these functions inductively. When $m = 1$, we have $\mu(Z^1) < \epsilon < \mu(Z^*)$ and thus $\theta_1$ is obtained immediately from Lemma \ref{LEM SIMPLEMIX}. Now suppose that $\theta_1$ through $\theta_{m-1}$ have been defined. Then
$$Z^* \setminus \bigcup_{k = 1}^{m-1} \theta_k(Z^k)$$
lies in $\salg_G(\{Z^*, Z^1, \ldots, Z^{m-1} \})$ by Lemma \ref{LEM EXPMOVE}. By (\ref{EQN KAPPA}) we have
$$\mu(Z^m) < \epsilon < \mu(Z^*) - \sum_{k = 1}^{m-1} \mu(Z^k) = \mu \left( Z^* \setminus \bigcup_{k = 1}^{m-1} \theta_k(Z^k) \right).$$
Therefore we may apply Lemma \ref{LEM SIMPLEMIX} to obtain $\theta_m$. This completes the construction.

Define $f : \bigcup_{m = 1}^{2 n k} \rng(\theta_m) \rightarrow \{0, 1, \ldots, 2 n k - 1\}$ by setting $f(\theta_m(z)) = \ell$ for $z \in Z_j^m$ if and only if $\psi_j^\ell(z) \in S_j$, and
$$| \{0 \leq i \leq \ell \: \psi_j^i(z) \in S_j \} | = m.$$
For $i, t \in \N$ we let $\mathbb{B}_i(t) \in \{0, 1\}$ denote the $i^\text{th}$ digit in the binary expansion of $t$ (so $\mathbb{B}_i(t) = 0$ for all $i > \log_2(t) + 1$). Now define a Borel set $B_1 \subseteq X$ by the rule
$$x \in B_1 \Longleftrightarrow \begin{cases}
x \in W \cdot R & \text{or} \\
x \in c \cdot R & \text{or} \\
x \in q_1 \cdot R & \text{or} \\
x \in q_3 \cdot Z & \text{or} \\
x \in q_4 \cdot Z_1 & \text{or} \\
x \in q_5 \cdot Z^1 & \text{or} \\
x \in q_6 \cdot \theta_m(Z^{m+1}) & \text{for some } 1 \leq m < 2 n k, \text{ or} \\
x = q_6 \cdot \rho^i(z) & \text{where } 1 \leq i < k, \ z \in \dom(f), \\ & \qquad\qquad \text{and } \mathbb{B}_i(f(z)) = 1.
\end{cases}$$
It is important to note that $B_1 \subseteq (W \cup \{c\} \cup Q) \cdot R$. In particular, $B_1 \subseteq F \cdot R$ by Lemma \ref{LEM TECH}.(iii). We also define the Borel set
$$B_0 = M \setminus B_1 \supseteq (D \setminus F \cdot R) \cup \Big( F \cdot R \setminus (W \cup \{c\} \cup Q) \cdot R \Big) \cup \Big( (W \cup \{c\} \cup Q) \cdot R \setminus B_1 \Big).$$
Note that clauses (iii) and (iv) of Lemma \ref{LEM TECH} imply that for every $r \neq r' \in R$
$$(W \cup \{c\} \cup Q) \cdot r \cap (W \cup \{c\} \cup Q) \cdot r' = \varnothing.$$
Thus from clause (ii) of Lemma \ref{LEM TECH} we obtain the following one-way implications
$$x \in B_0 \Longleftarrow \begin{cases}
x \in q_2 \cdot R & \text{or} \\
x \in q_3 \cdot (R \setminus Z) & \text{or} \\
x \in q_4 \cdot (R \setminus Z_1) & \text{or} \\
x \in q_5 \cdot (R \setminus Z^1) & \text{or} \\
x \in q_6 \cdot \textstyle{\bigcap_{m = 1}^{2 n k - 1} (Z \setminus \theta_m(Z^{m+1}))} & \text{or} \\
x = q_6 \cdot \rho^i(z) & \text{where } 1 \leq i < k, \ z \in \dom(f), \\ & \qquad\qquad \text{and } \mathbb{B}_i(f(z)) = 0.
\end{cases}$$
In particular, $q_2 \cdot R \subseteq B_0$. Therefore $\beta = \{B_0, B_1\}$ satisfies the assumptions of Lemma \ref{LEM REC}.

We will now check that $S_1, S_2 \in \ralg_G(\beta)$. By Lemma \ref{LEM REC} we have $R \in \ralg_G(\beta)$. By $G$-invariance of $\ralg_G(\beta)$, we have $q_i \cdot R \in \ralg_G(\beta)$ for $1 \leq i \leq 6$. Since $q_i \cdot R \subseteq B_0 \cup B_1$, it immediately follows from the definition of $\ralg_G(\beta)$ that $B_0 \cap q_i \cdot R$ and $B_1 \cap q_i \cdot R$ lie in $\ralg_G(\beta)$. Defining the partition
$$\gamma = \Big\{R, X \setminus (R \cup Q \cdot R) \Big\} \cup \Big\{B_0 \cap q_i \cdot R : 1 \leq i \leq 6 \Big\} \cup \Big\{B_1 \cap q_i \cdot R : 1 \leq i \leq 6 \Big\},$$
we have $\salg_G(\gamma) \subseteq \ralg_G(\beta)$. It suffices to show that $S_1, S_2 \in \salg_G(\gamma)$.

We have $x \in Z$ if and only if $q_3 \cdot x \in B_1 \cap q_3 \cdot R \in \gamma$. Thus $Z \in \salg_G(\gamma)$. Similarly, $x \in Z_1$ if and only if $q_4 \cdot x \in B_1 \cap q_4 \cdot R$. As $R \in \gamma$, we conclude that $R, Z, Z_1, Z_2 = Z \setminus Z_1 \in \salg_G(\gamma)$. It follows that $\rho$, $\psi_1$, and $\psi_2$ are $\salg_G(\gamma)$-expressible.

We prove by induction on $1 \leq m \leq 2 n k$ that $Z^m, Z_1^m, Z_2^m \in \salg_G(\gamma)$ and that $\theta_m$ is $\salg_G(\gamma)$-expressible. Since $x \in Z^1$ if and only if $q_5 \cdot x \in B_1 \cap q_5 \cdot R$, we have $Z^1 \in \salg_G(\gamma)$. Also $Z_1^1 = Z^1 \cap Z_1$ and $Z_2^1 = Z^1 \cap Z_2$ are in $\salg_G(\gamma)$. So $Z^* = Z \setminus Z^1$, $Z_1^* = Z_1 \setminus Z_1^1$, and $Z_2^* = Z_2 \setminus Z_2^1$ are in $\salg_G(\gamma)$ as well. It follows that $\theta_1$ is $\salg_G(\gamma)$-expressible. Now inductively suppose that $Z^i \in \salg_G(\gamma)$ and that $\theta_i$ is $\salg_G(\gamma)$-expressible for all $1 \leq i \leq m$. Then $z \in Z^{m+1}$ if and only if $z \in Z^m$ and $q_6 \cdot \theta_m(z) \in B_1 \cap q_6 \cdot R$. In other words,
$$Z^{m+1} = \theta_m^{-1} \Big(q_6^{-1} \cdot (B_1 \cap q_6 \cdot R) \Big).$$
Thus $Z^{m+1} \in \salg_G(\gamma)$ by Lemmas \ref{LEM EXPMOVE} and \ref{LEM EXPGROUP}. Similarly, $Z_1^{m+1} = Z^{m+1} \cap Z_1$ and $Z_2^{m+1} = Z^{m+1} \cap Z_2$ are in $\salg_G(\gamma)$. Finally, $\theta_{m+1}$ is expressible with respect to $\salg_G(\{Z^*, Z^1, \ldots, Z^{m+1}\}) \subseteq \salg_G(\gamma)$. This completes the inductive argument.

Now to complete the proof we show that $S_1, S_2 \in \salg_G(\gamma)$. We first argue that $f$ is $\salg_G(\gamma)$-measurable. It follows from the previous paragraph and Lemma \ref{LEM EXPMOVE} that $\dom(f) \in \salg_G(\gamma)$. Observe that the numbers $\ell \in \rng(f)$ are distinguished by their first $(k - 1)$-binary digits $\mathbb{B}_i(\ell)$, $1 \leq i < k$, since by construction $\log_2(2 n k) < k - 1$. So for $0 \leq \ell < 2 n k$, if we set $I_0 = \{1 \leq i < k \: \mathbb{B}_i(\ell) = 0\}$ and $I_1 = \{1 \leq i < k\} \setminus I_0$ then we have
$$f^{-1}(\ell) = \dom(f) \cap \bigcap_{i \in I_0} \rho^{-i} \Big( q_6^{-1} \cdot (B_0 \cap q_6 \cdot R) \Big) \cap \bigcap_{i \in I_1} \rho^{-i} \Big( q_6^{-1} \cdot (B_1 \cap q_6 \cdot R) \Big).$$
Thus $f^{-1}(\ell) \in \salg_G(\gamma)$ by Lemmas \ref{LEM EXPMOVE} and \ref{LEM EXPGROUP}. Now suppose that $x \in S_j$. Then there is $z \in Z_j$ and $0 \leq \ell < 2 n k$ with $x = \psi_j^\ell(z)$. It follows that $z \in Z_j^m$ where
$$m = |\{0 \leq i \leq \ell \: \psi_j^i(z) \in S_j\}|.$$
Furthermore, $\ell = f(\theta_m(z))$. Conversely, if there is $1 \leq m \leq 2 n k$, $z \in Z_j^m$, and $0 \leq \ell < 2 n k$ with $x = \psi_j^\ell(z)$ and $f(\theta_m(z)) = \ell$, then $x \in S_j$. Therefore
\begin{equation*}
S_j = \bigcup_{\ell = 0}^{2 n k - 1} \bigcup_{m = 1}^{2 n k} \psi_j^\ell \Big( Z_j \cap \theta_m^{-1}(f^{-1}(\ell)) \Big) \in \salg_G(\gamma) \subseteq \ralg_G(\beta). \qedhere
\end{equation*}
\end{proof}

\section{Krieger's finite generator theorem} \label{SECT KRIEGER}

We now present the main theorem.

\begin{thm} \label{THM RELBASIC}
Let $G \acts (X, \mu)$ be a {\pmp} ergodic action with $(X, \mu)$ non-atomic, and let $\cF$ be a $G$-invariant sub-$\sigma$-algebra. If $\xi$ is a countable Borel partition of $X$, $0 < r \leq 1$, and $\pv$ is a probability vector with $\sH(\xi | \cF) < r \cdot \sH(\pv)$, then there is a Borel pre-partition $\alpha = \{A_i : 0 \leq i < |\pv|\}$ with $\mu(A_i) = r p_i$ for every $i$ and $\salg_G(\xi) \vee \cF \subseteq \ralg_G(\alpha) \vee \cF$.
\end{thm}

\begin{proof}
Apply Proposition \ref{PROP FINGEN} to obtain a finite Borel partition $\xi'$ with $\salg_G(\xi') \vee \cF = \salg_G(\xi) \vee \cF$ and $\sH(\xi' | \cF) < r \cdot \sH(\pv)$. Since $\xi'$ is finite, by Lemma \ref{LEM SHAN} we have that $\sH(\xi' | \cF)$ is equal to the infimum of $\sH(\xi' | \zeta)$ over finite $\cF$-measurable partitions $\zeta$ of $X$. So fix a finite $\cF$-measurable partition $\zeta$ with $\sH(\xi' | \zeta) < r \cdot \sH(\pv)$. Since $\sH(\xi' \vee \zeta | \zeta) = \sH(\xi' | \zeta)$ and $\salg_G(\xi' \vee \zeta) \vee \cF = \salg_G(\xi') \vee \cF$, we may replace $\xi'$ with $\xi' \vee \zeta$ if necessary and assume that $\xi'$ refines $\zeta$. Let $\pi_\zeta : \xi' \rightarrow \zeta$ be the coarsening map. Finally, by Lemma \ref{LEM SHAN} we may let $\qv$ be a finite probability vector which coarsens $\pv$ and satisfies $\sH(\xi' | \zeta) < r \cdot \sH(\qv) \leq r \cdot \sH(\pv)$. Since $\sH(\qv) > 0$, by permuting the coordinates of $\qv$ if necessary we may assume that $0 < q_0 \leq q_1$. 

Fix $\kappa > 0$ with $\sH(\xi' | \zeta) + \kappa < r \sH(\qv) - r \kappa$. Apply Corollary \ref{COR SEP} to $\qv, \kappa$ to obtain $\delta_0, \epsilon_0 > 0$ with the property that for all $0 < \epsilon < \epsilon_0$ and all sufficiently large $n$
\begin{equation} \label{eqn:sep}
\exists K \subseteq L_{\qv,\epsilon}^n \quad |K| \geq \exp(n \cdot \sH(\qv) - n \cdot \kappa) \quad \text{and} \quad \forall k \neq k' \in K \ \dHam(k, k') > 2 \delta_0.
\end{equation}
By Lemma \ref{LEM RELSTIRLING} we may shrink $\epsilon_0$ so that for all $0 < \epsilon < \epsilon_0$, all sufficiently large $n$, and all $z \in L_{\zeta, \epsilon}^n$
\begin{equation} \label{eqn:relstirling}
|\{c \in L_{\xi',\epsilon}^n : \pi_\zeta(c) = z\}| \leq \exp(n \cdot \sH(\xi' | \zeta) + n \cdot \kappa).
\end{equation}
Set $\delta = r \cdot q_0 \cdot \delta_0 / 6$. Let $M$ with $\mu(M) = \delta$ and $\epsilon > 0$ be given by Proposition \ref{PROP CODE}. By shrinking $\epsilon$ if necessary, we may assume that $\epsilon < \epsilon_0$ and $\epsilon |\qv| < \delta_0 / 6$. Let $n_0 \in \N$ be such that for all $n \geq \lfloor r n_0 \rfloor$: statement (\ref{eqn:sep}) holds, for all $z \in L_{\zeta, \epsilon}^n$ inequality (\ref{eqn:relstirling}) holds,
$$(\epsilon n + 1) \cdot |\qv| < (\delta_0 / 6) n, \quad \text{and}$$
\begin{equation} \label{EQN KAPPA2}
n \cdot \sH(\xi' | \zeta) + n \kappa < \lfloor r n \rfloor \cdot \sH(\qv) - \lfloor r n \rfloor \cdot \kappa.
\end{equation}

By Proposition \ref{PROP RLEM} there are $n \geq n_0$, Borel sets $S_1, S_2 \subseteq X$ with $\mu(S_1) + \mu(S_2) < \epsilon$, and a $\salg_G(\{S_1, S_2\})$-expressible $\theta \in [E_G^X]$ such that $E_\theta$ admits a $\salg_G(\{S_1, S_2\})$-measurable transversal $Y$ and such that for $\mu$-almost-every $x \in X$, the $E_\theta$ class of $x$ has cardinality $n$,
\begin{align*}
& \forall C \in \xi' \cup \zeta  & \mu(C) - \epsilon < \frac{|C \cap [x]_{E_\theta}|}{|[x]_{E_\theta}|} & < \mu(C) + \epsilon, \\
& \text{and} & \qquad \frac{|M \cap [x]_{E_\theta}|}{|[x]_{E_\theta}|} & < \mu(M) + \delta = 2 \delta.
\end{align*}
So we have $\nm_{\xi'}^\theta(y) \in L_{\xi', \epsilon}^n$ and $\nm_\zeta^\theta(y) \in L_{\zeta, \epsilon}^n$ for almost-every $y \in Y$.

We set $m = \lfloor r \cdot n \rfloor$ and encourage the reader to pay attention to the distinction between $m$ and $n$. Let $K \subseteq L_{\qv, \epsilon}^m$ be as given by (\ref{eqn:sep}) so that $\dHam(k, k') > 2 \delta_0$ for all $k \neq k' \in K$. By (\ref{eqn:sep}), (\ref{eqn:relstirling}), and (\ref{EQN KAPPA2}) we have that for every $z \in L_{\zeta, \epsilon}^n$
\begin{equation*}
\Big| \{c \in L_{\xi', \epsilon}^n : \pi_\zeta(c) = z\} \Big| \leq \exp \Big( n \cdot \sH(\xi' | \zeta) + n \cdot \kappa \Big) < \exp \Big( m \cdot \sH(\qv) - m \cdot \kappa \Big) \leq |K|.
\end{equation*}
Thus for every $z \in L_{\zeta, \epsilon}^n$ we may fix an injection $f_z : \{c \in L_{\xi', \epsilon}^n : \pi_\zeta(c) = z\} \rightarrow K \subseteq L_{\qv, \epsilon}^m$.

For $y \in Y$ set $z_y = \nm_\zeta^\theta(y)$, $c_y = \nm_{\xi'}^\theta(y)$, and $\tilde{a}_y = f_{z_y}(c_y) \in K \subseteq L_{\qv, \epsilon}^m$. Also define $M_y = \{0 \leq i < n \: \theta^i(y) \in M\}$. Then
$$|M_y| < 2 \delta \cdot n = 2 (q_0 \cdot \delta_0 / 6) r n < (\delta_0 / 3) (m + 1) \leq (2 \delta_0 / 3) m$$
for $\mu$-almost-every $y \in Y$. Since $\tilde{a}_y \in L_{\qv, \epsilon}^m$, Lemma \ref{LEM J} provides a set $J_y \subseteq \{0, 1, \ldots, m-1\}$ with $|J_y| < (\epsilon m + 1) \cdot |\qv| + (\delta_0/6) m < (\delta_0 / 3) m$ such that for all $0 \leq t < |\qv|$
\begin{equation} \label{eqn:dist}
\frac{1}{m} \cdot \Big| \tilde{a}_y^{-1}(t) \setminus (M_y \cup J_y) \Big| \leq \frac{1}{m}  \cdot \Big| \tilde{a}_y^{-1}(t) \setminus J_y \Big| < (1 - \delta_0 / 6) q_t.
\end{equation}
Since there are only finitely many choices for $J_y$, it is easy to arrange the map $y \mapsto J_y$ to be Borel. We then let $J$ be the Borel set $J = \{\theta^j(y) \: y \in Y, \ j \in J_y\}$.

Define the pre-partition $\alpha^0 = \{A_t^0 \: 0 \leq t < |\qv|\}$ by setting
$$A_t^0 = \{\theta^i(y) \: y \in Y, \ 0 \leq i < m, \ i \not\in M_y \cup J_y, \text{ and } \tilde{a}_y(i) = t\}.$$
Observe that $\mu(Y) = 1 / n$ since $Y$ is a transversal for $E_\theta$. By (\ref{eqn:dist}) we have that for every $0 \leq t < |\qv|$
\begin{align*}
\mu(A_t^0) & = \int_Y |\tilde{a}_y^{-1}(t) \setminus (M_y \cup J_y)| \ d \mu(y)\\
 & < \frac{m}{n} \cdot (1 - \delta_0/6) q_t \leq r (1 - \delta_0 / 6)\cdot q_t = \left(r - \frac{\delta}{q_0} \right) q_t.
\end{align*}
Thus $\mu(A_t^0) < r \cdot q_t$ and, since $q_0 \leq q_1$, we have $\mu(A_0^0) < r \cdot q_0 - \delta = r \cdot q_0 - \mu(M)$ and similarly $\mu(A_1^0) \leq r \cdot q_1 - \mu(M)$.

Apply Proposition \ref{PROP CODE} to get a partition $\beta = \{B_0, B_1\}$ of $M$ with $S_1, S_2 \in \ralg_G(\beta)$. Since $M$ is disjoint from $\cup \alpha^0$ and $\mu$ is non-atomic, there is a pre-partition $\alpha' = \{A_t' : 0 \leq t < |\qv|\}$ with $\mu(A_t') = r \cdot q_t$ for every $0 \leq t < |\qv|$, $A_0^0 \cup B_0 \subseteq A_0'$, $A_1^0 \cup B_1 \subseteq A_1'$, and $A_t^0 \subseteq A_t'$ for $2 \leq t < |\qv|$. The pre-partition $\alpha'$ extends both $\alpha^0$ and $\beta$. In particular, $S_1, S_2 \in \ralg_G(\beta) \subseteq \ralg_G(\alpha')$ by Lemma \ref{LEM EXT}. We have that $\theta$ is expressible and $Y$ is measurable with respect to $\salg_G(\{S_1, S_2\}) \subseteq \ralg_G(\alpha')$. By Lemma \ref{LEM EXPGROUP} it follows that $\theta^i$ is $\ralg_G(\alpha')$-expressible for all $i \in \Z$. We will show that $\xi' \subseteq \ralg_G(\alpha') \vee \cF$.

We claim that the map $y \in Y \mapsto \tilde{a}_y$ is $\ralg_G(\alpha')$-measurable. We check this via the definition of reduced $\sigma$-algebras. Fix $y \in Y$ and $x \in X$ with either $x \not\in Y$ or $\tilde{a}_x \neq \tilde{a}_y$. If $x \not\in Y$ then we are done since $Y \in \ralg_G(\alpha')$ (since then there is $g \in G$ with $g \cdot x, g \cdot y \in \cup \alpha'$ and with $\alpha'$ separating $g \cdot x$ from $g \cdot y$). So suppose that $x \in Y$ and $\tilde{a}_x \neq \tilde{a}_y$. Then $\dHam(\tilde{a}_y, \tilde{a}_x) > 2 \delta_0$ since $\tilde{a}_x, \tilde{a}_y \in K$. Set $I = \{0 \leq i < m \: \tilde{a}_y(i) \neq \tilde{a}_x(i)\}$ and note $|I| > 2 \delta_0 \cdot m$. Since
$$\Big| M_y \cup J_y \cup M_x \cup J_x \Big| < 2 \cdot (2 \delta_0 / 3) m + 2 \cdot (\delta_0 / 3) m = 2 \delta_0 m,$$
we may fix $i \in I \setminus (M_y \cup J_y \cup M_x \cup J_x)$. Since $\theta^i$ is $\ralg_G(\alpha')$-expressible, there is a $\ralg_G(\alpha')$-measurable partition $\{Z_g : g \in G\}$ of $X$ such that $\theta^i(z) = g \cdot z$ for all $g \in G$ and $z \in Z_g$. If $y$ and $x$ are separated by the partition $\{Z_g : g \in G\}$ then, since this partition is $\ralg_G(\alpha')$-measurable, there must be $h \in G$ with both $h \cdot y$ and $h \cdot x$ lying in $\cup \alpha'$ and separated by $\alpha'$. We are done in this case. So assume there is $g \in G$ with $y, x \in Z_g$. Then $g \cdot y = \theta^i(y)$ lies in $A_t^0 \subseteq A_t'$ where $t = \tilde{a}_y(i)$ and similarly $g \cdot x = \theta^i(x)$ lies in $A_s^0 \subseteq A_s'$ where $s = \tilde{a}_x(i)$. As $t \neq s$ we have that $g \cdot y$ and $g \cdot x$ lie in $\cup \alpha'$ and are separated by $\alpha'$. This proves the claim.

We observe that the map $y \in Y \mapsto z_y$ is $\ralg_G(\alpha') \vee \cF$-measurable since $\zeta \subseteq \cF$ and the value of $z_y$ is entirely determined by the location of $y$ in the partition $\bigvee_{i=0}^{n-1} \theta^{-i}(\zeta) \res Y$ of $Y$. This partition is $\ralg_G(\alpha') \vee \cF$-measurable by Lemmas \ref{LEM EXPMOVE} and \ref{LEM EXPGROUP}. So the map $y \in Y \mapsto (z_y, \tilde{a}_y)$ is $\ralg_G(\alpha') \vee \cF$-measurable. Since $c_y = f_{z_y}^{-1}(\tilde{a}_y)$, it follows that the map $y \in Y \mapsto c_y$ is $\ralg_G(\alpha') \vee \cF$-measurable as well. For $C \in \xi'$ we have
\begin{equation*}
C = \{\theta^i(y) \: y \in Y, \ 0 \leq i < n, \text{ and } c_y(i) = C\} = \bigcup_{i = 0}^{n-1} \theta^i \Big( \{y \in Y \: c_y(i) = C\} \Big).
\end{equation*}
Therefore $\xi' \subseteq \ralg_G(\alpha') \vee \cF$ by Lemmas \ref{LEM EXPMOVE} and \ref{LEM EXPGROUP}. We conclude that
$$\salg_G(\xi) \vee \cF = \salg_G(\xi') \vee \cF \subseteq \ralg_G(\alpha') \vee \cF.$$
Finally, since $(X, \mu)$ is non-atomic, $\mu(A_t') = r \cdot q_t$, and $\qv$ is a coarsening of $\pv$, there is a refinement $\alpha$ of $\alpha'$ with $\mu(A_t) = r \cdot p_t$ for all $0 \leq t < |\pv|$. Clearly we still have $\salg_G(\xi) \vee \cF \subseteq \ralg_G(\alpha) \vee \cF$.
\end{proof}

Note that Theorem \ref{INTRO THM2} follows from the above theorem by choosing a partition $\xi$ with $\sH(\xi | \cF)$ close to $\rh_G(X, \mu | \cF)$ and with $\salg_G(\xi) \vee \cF = \mathcal{B}(X)$.

\begin{cor} \label{COR ADD}
Let $G \acts (X, \mu)$ be a {\pmp} ergodic action with $(X, \mu)$ non-atomic, and let $\cF$ be a $G$-invariant sub-$\sigma$-algebra. If $G \acts (Y, \nu)$ is a factor of $G \acts (X, \mu)$ and $\Sigma$ is the sub-$\sigma$-algebra of $X$ associated to $Y$ then
$$\rh_G(X, \mu | \cF) \leq \rh_G(Y, \nu) + \rh_G(X, \mu | \cF \vee \Sigma).$$
\end{cor}

\begin{proof}
This is immediate if either $\rh_G(Y, \nu)$ or $\rh_G(X, \mu | \cF \vee \Sigma)$ is infinite, so suppose that both are finite. Fix $\epsilon > 0$ and fix a generating partition $\beta'$ for $G \acts (Y, \nu)$ with $\sH(\beta') < \rh_G(Y, \nu) + \epsilon / 2$. Pull back $\beta'$ to a partition $\beta$ of $X$. Then $\sH(\beta) = \sH(\beta')$ and $\salg_G(\beta) = \Sigma$. By definition of $\rh_G(X, \mu | \cF \vee \Sigma)$, there is a partition $\gamma'$ of $X$ with
$$\sH(\gamma' | \cF \vee \Sigma) < \rh_G(X, \mu | \cF \vee \Sigma) + \epsilon / 2$$
and $\salg_G(\gamma') \vee \cF \vee \Sigma = \mathcal{B}(X)$. Apply Theorem \ref{THM RELBASIC} to get a partition $\gamma$ of $X$ with
$$\sH(\gamma) < \rh_G(X, \mu | \cF \vee \Sigma) + \epsilon / 2$$
and $\salg_G(\gamma) \vee \cF \vee \Sigma = \mathcal{B}(X)$. Then
$$\mathcal{B}(X) = \salg_G(\gamma) \vee \cF \vee \Sigma = \salg_G(\gamma \vee \beta) \vee \cF,$$
and hence
\begin{equation*}
\rh_G(X, \mu | \cF) \leq \sH(\beta \vee \gamma | \cF) \leq \sH(\beta) + \sH(\gamma) < \rh_G(Y, \nu) + \rh_G(X, \mu | \cF \vee \Sigma) + \epsilon. \qedhere
\end{equation*}
\end{proof}

\section{Relative Rokhlin entropy and amenable groups} \label{SECT AMENABLE}

We verify that for free ergodic actions of amenable groups, relative Rokhlin entropy and relative Kolmogorov--Sinai entropy agree. This result was previously established in the non-relative case by the author and Tucker-Drob \cite{ST14}.

We first recall the definition of relative Kolmogorov--Sinai entropy. Let $G$ be a countably infinite amenable group, and let $G \acts (X, \mu)$ be a free {\pmp} action. For a partition $\alpha$ and a finite set $T \subseteq G$, we write $\alpha^T$ for the join $\bigvee_{t \in T} t \cdot \alpha$, where $t \cdot \alpha = \{t \cdot A : A \in \alpha\}$. Given a $G$-invariant sub-$\sigma$-algebra $\cF$, the relative Kolmogorov--Sinai entropy is defined as
$$\ksh_G(X, \mu | \cF) = \sup_\alpha \inf_{T \subseteq G} \frac{1}{|T|} \cdot \sH(\alpha^T | \cF),$$
where $\alpha$ ranges over all finite partitions and $T$ ranges over finite subsets of $G$ \cite{DP02}. Equivalently, one can replace the infimum with a limit over a F{\o}lner sequence $(T_n)$ \cite{OW87}. Recall that a sequence $T_n \subseteq G$ of finite sets is a \emph{F{\o}lner sequence} if
$$\lim_{n \rightarrow \infty} \frac{|\Bnd{K}{T_n}|}{|T_n|} = 0$$
for every finite $K \subseteq G$, where $\Bnd{K}{T} = \{t \in T : t K \not\subseteq T\}$. We also write $\Int{K}{T}$ for $T \setminus \Bnd{K}{T}$.

\begin{prop} \label{PROP RELROK}
Let $G$ be a countably infinite amenable group, let $G \acts (X, \mu)$ be a free ergodic action, and let $\cF$ be a $G$-invariant sub-$\sigma$-algebra. Then the relative Kolmogorov--Sinai entropy and relative Rokhlin entropy coincide:
$$\ksh_G(X, \mu | \cF) = \rh_G(X, \mu | \cF).$$
\end{prop}

\begin{proof}
We first show that $\ksh_G(X, \mu | \cF) \leq \rh_G(X, \mu | \cF)$. If $\rh_G(X, \mu | \cF) = \infty$ then there is nothing to show. So suppose that $\rh_G(X, \mu | \cF) < \infty$ and fix $\epsilon > 0$. Let $\alpha$ be a countable partition with $\salg_G(\alpha) \vee \cF = \mathcal{B}(X)$ and $\sH(\alpha | \cF) < \rh_G(X, \mu | \cF) + \epsilon$. Let $\beta$ be any finite partition of $X$ and let $(T_n)$ be a F{\o}lner sequence. Then by Lemma \ref{LEM SHAN}
$$0 = \sH(\beta | \salg_G(\alpha) \vee \cF) = \inf_{K \subseteq G} \sH(\beta | \alpha^K \vee \cF),$$
where $K$ ranges over finite subsets of $G$. Fix $K \subseteq G$ so that $\sH(\beta | \alpha^K \vee \cF) < \epsilon$. Note that $\sH(t \cdot \beta | \alpha^{t K} \vee \cF) < \epsilon$ for all $t \in G$. Therefore
\begin{align*}
\lim_{n \rightarrow \infty} \frac{1}{|T_n|} & \cdot \sH(\beta^{T_n} | \cF)\\
 & \leq \lim_{n \rightarrow \infty} \frac{1}{|T_n|} \cdot \sH(\alpha^{T_n} \vee \beta^{T_n} | \cF)\\
 & = \lim_{n \rightarrow \infty} \frac{1}{|T_n|} \cdot \sH(\alpha^{T_n} | \cF) + \frac{1}{|T_n|} \cdot \sH(\beta^{T_n} | \alpha^{T_n} \vee \cF)\\
 & \leq \lim_{n \rightarrow \infty} \frac{1}{|T_n|} \cdot \sum_{t \in T_n} \Big( \sH(t \cdot \alpha | \cF) + \sH(t \cdot \beta | \alpha^{T_n} \vee \cF) \Big)\\
 & < \lim_{n \rightarrow \infty} \rh_G(X, \mu | \cF) + \epsilon + \frac{|\Int{K}{T_n}|}{|T_n|} \cdot \epsilon + \frac{|\Bnd{K}{T_n}|}{|T_n|} \cdot \sH(\beta)\\
 & = \rh_G(X, \mu | \cF) + 2 \epsilon.
\end{align*}
Now let $\epsilon$ tend to $0$ and then take the supremum over all $\beta$.

Now we argue that $\rh_G(X, \mu | \cF) \leq \ksh_G(X, \mu | \cF)$. Again this is immediate if $\ksh_G(X, \mu | \cF) = \infty$, so we assume $\ksh_G(X, \mu | \cF) < \infty$. Since the action of $G$ is free, a theorem of Seward and Tucker-Drob \cite{ST14} states that there is a factor action $G \acts (Z, \eta)$ of $(X, \mu)$ such that the action of $G$ on $Z$ is free and $\rh_G(Z, \eta) < \epsilon$. Let $\Sigma$ be the $G$-invariant sub-$\sigma$-algebra of $X$ associated to $Z$, and let $G \acts (Y, \nu)$ be the factor of $(X, \mu)$ associated to $\cF \vee \Sigma$. Then $G$ acts freely on $(Y, \nu)$ since $(Y, \nu)$ factors onto $(Z, \eta)$. By the Ornstein--Weiss theorem \cite{OW80}, all free ergodic actions of countably infinite amenable groups are orbit equivalent. In particular, there is a free ergodic {\pmp} action $\Z \acts (Y, \nu)$ which has the same orbits as $G \acts (Y, \nu)$ and has $0$ Kolmogorov--Sinai entropy, $\ksh_\Z(Y, \nu) = 0$. By the Rokhlin generator theorem \cite{Roh67}, we have $\rh_\Z(Y, \nu) = 0$ as well.

Let's say $\Z = \langle t \rangle$. Define $c : Y \rightarrow G$ by
$$c(y) = g \Longleftrightarrow t \cdot y = g \cdot y.$$
Let $f : (X, \mu) \rightarrow (Y, \nu)$ be the factor map, and let $\Z$ act on $(X, \mu)$ by setting
$$t \cdot x = c(f(x)) \cdot x.$$
Then $\cF \vee \Sigma$ and the actions of $G$ and $\Z$ on $(X, \mu)$ satisfy the assumptions of Proposition \ref{PROP OE}. Equivalently, in the terminology of Rudolph--Weiss \cite{RW00}, the orbit-change cocycles between the actions of $G$ and $\Z$ on $X$ are $\cF \vee \Sigma$-measurable. Thus $\ksh_G(X, \mu | \cF \vee \Sigma) = \ksh_\Z(X, \mu | \cF \vee \Sigma)$ by \cite[Theorem 2.6]{RW00}. Also, since $\rh_\Z(X, \mu | \cF \vee \Sigma) \leq \rh_\Z(X, \mu)$ and $\rh_\Z(Y, \nu) = 0$, it follows from Corollary \ref{COR ADD} that
\begin{equation} \label{eqn:Z}
\rh_\Z(X, \mu | \cF \vee \Sigma) = \rh_\Z(X, \mu).
\end{equation}

We have
\begin{equation*}
\begin{array}{rcll}
\ksh_G(X, \mu | \cF \vee \Sigma) & = & \ksh_\Z(X, \mu | \cF \vee \Sigma) & \text{by the Rudolph--Weiss theorem \cite{RW00}}\\
                  & = & \ksh_\Z(X, \mu) - \ksh_\Z(Y, \nu) & \text{by the Abramov--Rokhlin theorem \cite{AR62}}\\
                  & = & \ksh_\Z(X, \mu) & \text{since } \ksh_\Z(Y, \nu) = 0\\
                  & = & \rh_\Z(X, \mu) & \text{by the Rokhlin generator theorem \cite{Roh67}}\\
                  & = & \rh_\Z(X, \mu | \cF \vee \Sigma) & \text{by Equation \ref{eqn:Z}}\\
                  & = & \rh_G(X, \mu | \cF \vee \Sigma) & \text{by Proposition \ref{PROP OE}}
\end{array}
\end{equation*}
So $\ksh_G(X, \mu | \cF \vee \Sigma) = \rh_G(X, \mu | \cF \vee \Sigma)$. Also, it is immediate from the definitions that $\ksh_G(X, \mu | \cF \vee \Sigma) \leq \ksh_G(X, \mu | \cF)$. Finally, by Corollary \ref{COR ADD} we have
$$\rh_G(X, \mu | \cF) \leq \rh_G(Z, \eta) + \rh_G(X, \mu | \cF \vee \Sigma) < \epsilon + \ksh_G(X, \mu | \cF \vee \Sigma) \leq \epsilon + \ksh_G(X, \mu | \cF).$$
Now let $\epsilon$ tend to $0$.
\end{proof}

\end{document}